\documentclass{article}      
\usepackage{amsthm}
\usepackage{amssymb} 
\usepackage{amsmath} 
\DeclareMathAccent{\vecTeX}{\mathord}{letters}{"7E} 
\begin{document}
\def\qed{\hfill$\Box$}
\newtheorem{theorem}{Theorem}{\it}{\rm}
\newtheorem{corollary}{Corollary}{\it}{\rm}
\newtheorem{lemma}{Lemma}{\it}{\rm}
\newtheorem{proposition}{Proposition}{\it}{\rm} 
\makeatletter
 \@addtoreset{corollary}{section}
 \@addtoreset{equation}{section}
 \@addtoreset{examples}{section}
 \@addtoreset{lemma}{section}
 \@addtoreset{proposition}{section}
 \@addtoreset{remarks}{section}
 \@addtoreset{theorem}{section}
\makeatother
\theoremstyle{definition} 
\newtheorem{examples}{Examples}  
\newtheorem{remark}{Remark} 
\newtheorem{remarks}{Remarks} 
\renewcommand{\thecorollary}{\arabic{section}.\arabic{corollary}}        
\renewcommand{\theequation}{\arabic{section}.\arabic{equation}}        
\renewcommand{\theexamples}{\arabic{section}.\arabic{examples}}        
\renewcommand{\thelemma}{\arabic{section}.\arabic{lemma}}        
\renewcommand{\theproposition}{\arabic{section}.\arabic{proposition}}        
\renewcommand{\theremark}{\arabic{section}.\arabic{remark}}        
\renewcommand{\theremarks}{\arabic{section}.\arabic{remarks}}        
\renewcommand{\thetheorem}{\arabic{section}.\arabic{theorem}}   
\makeatletter 
\def\@thm#1#2#3{%
  \ifhmode\unskip\unskip\par\fi
  \normalfont
  \trivlist
  \let\thmheadnl\relax
  \let\thm@swap\@gobble
  \thm@notefont{\fontseries\mddefault\upshape}%
  \thm@headpunct{}
  \thm@headsep 5\p@ plus\p@ minus\p@\relax
  \thm@space@setup
  #1
  \@topsep \thm@preskip               
  \@topsepadd \thm@postskip           
  \def\@tempa{#2}\ifx\@empty\@tempa
    \def\@tempa{\@oparg{\@begintheorem{#3}{}}[]}%
  \else
    \refstepcounter{#2}%
    \def\@tempa{\@oparg{\@begintheorem{#3}{\csname the#2\endcsname}}[]}%
  \fi
  \@tempa
}
\renewenvironment{proof}[1][\proofname]{\par
  \pushQED{\qed}%
  \normalfont \topsep6\p@\@plus6\p@\relax
  \trivlist
  \item[\hskip\labelsep
        \itshape
    #1\@addpunct{}]\ignorespaces
}{%
  \popQED\endtrivlist\@endpefalse
}
\providecommand{\proofname}{Proof} 
\makeatother 
 \newcommand*{\al}{\alpha}
 \newcommand*{\ba}{\beta}
 \newcommand*{\Da}{\Delta}
 \newcommand*{\da}{\delta}
 \newcommand*{\Ga}{\Gamma}
 \newcommand*{\ga}{\gamma}
 \newcommand*{\ia}{\iota}
 \newcommand*{\ka}{\kappa}
 \newcommand*{\wk}{{\wt{\ka}}}
 \newcommand*{\wkk}{{\wt{\ka}\ka}}
 \newcommand*{\kwk}{{\ka\wt{\ka}}}
 \newcommand*{\lda}{\lambda}
 \newcommand*{\na}{\nabla}
 \newcommand*{\Om}{\Omega}
 \newcommand*{\om}{\omega}
 \newcommand*{\oO}{\ol{\Omega}} 
 \newcommand*{\Sa}{\Sigma}
 \newcommand*{\sa}{\sigma}
 \newcommand*{\ta}{\theta}
 \newcommand*{\za}{\zeta}
 \newcommand*{\ve}{\varepsilon}
 \newcommand*{\vp}{\varphi}
 \newcommand*{\vt}{\vartheta}
 \newcommand*{\BA}{{\mathbb A}}
 \newcommand*{\BB}{{\mathbb B}}
 \newcommand*{\BC}{{\mathbb C}}
 \newcommand*{\BE}{{\mathbb E}}
 \newcommand*{\BF}{{\mathbb F}}
 \newcommand*{\BG}{{\mathbb G}}
 \newcommand*{\BH}{{\mathbb H}}
 \newcommand*{\BL}{{\mathbb L}}
 \newcommand*{\BN}{{\mathbb N}}
 \newcommand*{\BR}{{\mathbb R}}
 \newcommand*{\BX}{{\mathbb X}}
 \newcommand*{\cA}{{\mathcal A}}
 \newcommand*{\cB}{{\mathcal B}}
 \newcommand*{\cC}{{\mathcal C}}
 \newcommand*{\cD}{{\mathcal D}}
 \newcommand*{\cL}{{\mathcal L}}
 \newcommand*{\cR}{{\mathcal R}}
 \newcommand*{\cW}{{\mathcal W}}
 \newcommand*{\cX}{{\mathcal X}}
 \newcommand*{\cY}{{\mathcal Y}}
 \newcommand*{\gA}{{\mathfrak A}}
 \newcommand*{\gsa}{{\mathfrak a}}
 \newcommand*{\gB}{{\mathfrak B}}
 \newcommand*{\gF}{{\mathfrak F}}
 \newcommand*{\gK}{{\mathfrak K}}
 \newcommand*{\gM}{{\mathfrak M}}
 \newcommand*{\gN}{{\mathfrak N}}
 \newcommand*{\gss}{{\mathfrak s}}
 \newcommand*{\gsu}{{\mathfrak u}}
 \newcommand*{\gsv}{{\mathfrak v}}
 \newcommand*{\gW}{{\mathfrak W}}
 \newcommand*{\sA}{{\mathsf A}}
 \newcommand*{\sBC}{{\mathsf{BC}}}
 \newcommand*{\sW}{{\mathsf W}}
 \newcommand*{\bal}{\begin{aligned}}
 \newcommand*{\eal}{\end{aligned}}
 \newcommand*{\uti}[1]{(#1)\space\space}
 \newcommand*{\qa}{,\qquad}
 \newcommand*{\qb}{,\quad}
 \newcommand*{\mf}[1]{\boldsymbol{#1}}  
 \newcommand*{\ci}{\mathaccent"7017 }   
 \newcommand*{\hb}[1]{\hbox{$#1$}} 
 \newcommand*{\sdot}{\!\cdot\!}
 \newcommand*{\sco}{\kern2pt\colon\kern2pt}
 \newcommand*{\sn}{\kern1pt|\kern1pt}
 \newcommand*{\bsn}{\kern1pt\big|\kern1pt}
 \newcommand*{\Bsn}{\kern1pt\Big|\kern1pt}
 \newcommand*{\ssm}{\!\setminus\!}
 \newcommand*{\eesdot}{\hbox{$[\![{}\sdot{}]\!]$}}
 \newcommand*{\eea}[1]{\hbox{$[\![#1]\!]$}}
 \newcommand*{\beea}[1]{\hbox{$\big[\!\kern-1pt\big[#1\big]\!\kern-1pt\big]$}}
 \newcommand*{\Beea}[1]{\hbox{$\Big[\!\kern-1pt\Big[#1\Big]\!\kern-1pt\Big]$}}
 \newcommand*{\vsdot}{\hbox{$\vert\sdot\vert$}}
 \newcommand*{\Vsdot}{\hbox{$\Vert\sdot\Vert$}}
 \newcommand*{\npbd}{\postdisplaypenalty=20000} 
 \newcommand*{\po}{\pagebreak[2] }
 \newcommand*{\prsn}{\hbox{$(\cdot\sn\cdot)$}}
 \newcommand*{\pw}{\hbox{$\dl{}\sdot{},{}\sdot{}\dr$}}
 \newcommand*{\mfA}{{\mf{A}}} 
 \newcommand*{\mfsA}{{\mf{\sA}}} 
 \newcommand*{\mfB}{{\mf{B}}} 
 \newcommand*{\mfC}{{\mf{C}}} 
 \newcommand*{\mfE}{{\mf{E}}} 
 \newcommand*{\mff}{{\mf{f}}} 
 \newcommand*{\mfg}{{\mf{g}}} 
 \newcommand*{\mfh}{{\mf{h}}} 
 \newcommand*{\mfu}{{\mf{u}}} 
 \newcommand*{\mfv}{{\mf{v}}} 
 \newcommand*{\mfW}{{\mf{W}}} 
 \newcommand*{\mfsW}{{\mf{\sW}}} 
 \newcommand*{\mfw}{{\mf{w}}} 
 \newcommand*{\mfz}{{\mf{z}}} 
 \newcommand*{\mfvp}{{\mf{\vp}}} 
 \newcommand*{\mfpsi}{{\mf{\psi}}} 
 \newcommand*{\ol}{\overline}
 \newcommand*{\ul}{\underline}
 \newcommand*{\wh}{\widehat}
 \newcommand*{\wt}{\widetilde}
 \newsavebox{\Prel}
 \sbox{\Prel}{\begin{picture}(2,6)(0,0)\put(1,2.5){\circle*{2}}\end{picture}}
 \newcommand*{\btdot}{\mathrel{\usebox{\Prel}}}
 \newcommand*{\fdot}{{\usebox{\Prel}}}
 \newsavebox{\supiPrel} 
 \sbox{\supiPrel}{\begin{picture}(2,2)(0,0)
 \put(.9,.9){\circle*{1}}
 \put(1.1,.9){\circle*{1}}
 \put(1,.8){\circle*{1}}
 \put(1,1){\circle*{1}}
 \end{picture}}
 \newcommand*{\ifdot}{{\usebox{\supiPrel}}}
 \newcommand*{\vph}{\vphantom}
 \newcommand*{\hr}{\hookrightarrow}
 \newcommand*{\ra}{\rightarrow}
 \newcommand*{\dl}{\langle}
 \newcommand*{\dr}{\rangle}
 \newcommand*{\card}{\mathop{\rm card}\nolimits}
 \newcommand*{\diag}{{\mathop{\rm diag}\nolimits}}
 \newcommand*{\diam}{\mathop{\rm diam}\nolimits}
 \newcommand*{\tdiv}{\mathop{\rm div}\nolimits}
 \newcommand*{\grad}{\mathop{\rm grad}\nolimits}
 \newcommand*{\graph}{\mathop{\rm graph}\nolimits}
 \newcommand*{\id}{{\rm id}}
 \newcommand*{\supp}{\mathop{\rm supp}\nolimits}
 \newcommand*{\vol}{\mathop{\rm vol}\nolimits}
 \newcommand*{\Lis}{{\mathcal L}{\rm is}}
 \newcommand*{\loc}{{\rm loc}}
 \renewcommand*{\Re}{\mathop{\rm Re}\nolimits} 
 \newcommand*{\imi}{{\mit i\kern1pt}}
 \newcommand*{\mi}[1]{{#1\textrm-}}       
 \newcommand*{\is}{\subset}
 \newcommand*{\bt}{\bullet}
 \newcommand*{\es}{\emptyset}
 \newcommand*{\iy}{\infty}
 \newcommand*{\mt}{\mapsto}
 \newcommand*{\nat}{\natural}
 \newcommand*{\nag}{\na_{\cona\cona g}}
 \newcommand*{\naka}{\na_{\cona\ka}}
 \newcommand*{\naSa}{\na_{\cona\Sa}}
 \newcommand*{\pl}{\partial}
 \newcommand*{\pa}{\partial^\alpha}
 \newcommand*{\coc}{\kern.75pt}
 \newcommand*{\cona}{\kern-1pt}
 \newcommand*{\coU}{\kern-1pt}
 \newcommand*{\coW}{\kern-1pt}
 \newcommand*{\AB}{(\cA,\cB)}
 \newcommand*{\ArBr}{(\cA_r,\cB_r)}
 \newcommand*{\EE}{(E,E)}
 \newcommand*{\EF}{(E,F)}
 \newcommand*{\Hmgm}{(\BH^m,g_m)}
 \newcommand*{\Mg}{(M,g)}
 \newcommand*{\Migi}{(M_i,g_i)}
 \newcommand*{\whMwhg}{(\widehat M,\widehat g)}
 \newcommand*{\MV}{(M,V)}
 \newcommand*{\Nh}{(N,h)}
 \newcommand*{\oOgm}{(\oO,g_m)}
 \newcommand*{\RC}{(\BR,\BC)}
 \newcommand*{\Rmgm}{(\BR^m,g_m)}
 \newcommand*{\cRcRc}{(\cR,\cR^c)}
 \newcommand*{\Vg}{(V,g)}
 \newcommand*{\BXF}{(\BX,F)}
 \newcommand*{\BUC}{BU\kern-.3ex C}
 \makeatletter
 \newif\ifinany@
 \newcount\column@
 \def\column@plus{%
    \global\advance\column@\@ne
 }
 \newcount\maxfields@
 \def\add@amps#1{%
    \begingroup
        \count@#1
        \DN@{}%
        \loop
            \ifnum\count@>\column@
                \edef\next@{&\next@}%
                \advance\count@\m@ne
        \repeat
    \@xp\endgroup
    \next@
 }
 \def\Let@{\let\\\math@cr}
 \def\restore@math@cr{\def\math@cr@@@{\cr}}
 \restore@math@cr
 \def\default@tag{\let\tag\dft@tag}
 \default@tag
 \newbox\strutbox@
 \def\strut@{\copy\strutbox@}
 \addto@hook\every@math@size{%
  \global\setbox\strutbox@\hbox{\lower.5\normallineskiplimit
         \vbox{\kern-\normallineskiplimit\copy\strutbox}}}
 \renewcommand{\start@aligned}[2]{%
    \RIfM@\else
        \nonmatherr@{\begin{\@currenvir}}%
    \fi
    \null\,%
    \if #1t\vtop \else \if#1b \vbox \else \vcenter \fi \fi \bgroup
        \maxfields@#2\relax
        \ifnum\maxfields@>\m@ne
            \multiply\maxfields@\tw@
            \let\math@cr@@@\math@cr@@@alignedat
        \else
            \restore@math@cr
        \fi
        \Let@
        \default@tag
        \ifinany@\else\openup\jot\fi
        \column@\z@
        \ialign\bgroup
           &\column@plus
            \hfil
            \strut@
            $\m@th\displaystyle{##}$%
           &\column@plus
            $\m@th\displaystyle{{}##}$%
            \hfil
            \crcr
 }
 \renewenvironment{aligned}[1][c]{%
    \start@aligned{#1}\m@ne
 }{%
    \crcr\egroup\egroup
 }
 \makeatother
\title{Maximal Regularity of Parabolic Transmission Problems} 
\author{Herbert Amann} 
\date{}
\maketitle
\begin{center} 
{\small 
Dedicated to Matthias Hieber, a~pioneer of maximal regularity,\\  
on the occasion of his sixtieth birthday} 
\end{center} 
\begin{abstract} 
Linear reaction-diffusion equations with inhomogeneous boundary and
transmission conditions are shown to possess the property of maximal 
\hbox{$L_p$~regularity}. The new feature is the fact that the transmission 
interface is allowed to intersect the boundary of the domain transversally. 
\end{abstract}
\makeatletter
\def\blfootnote{\xdef\@thefnmark{}\@footnotetext}
\makeatother 
\setcounter{section}{0}
\section{Introduction}\label{sec-I} 
The emerging and understanding of the theory of maximal regularity for  
par\-a\-bol\-ic differential equations, which took place within the last 
three or 
so decades, has provided a firm basis for a successful handling of many 
challenging nonlinear problems. Among them, phase transition issues play a 
particularly prominent role. The impressive progress which has been made in 
this field with the help of maximal regularity techniques is 
well-documented in the book by J.~Pr{\"u}ss and G.~Simonett~\cite{PrS16a}.  
The reader may also consult the extensive list of references 
and the `Bibliographic Comments' in~\cite{PrS16a} for works of 
other authors and historical developments. 

\par 
The relevant mathematical setup is usually placed in the framework of 
parabolic equations in bounded Euclidean domains, the interface 
being modeled as a hypersurface. In most works known to the author it is 
assumed that the interface lies in the interior of the domain. Noteworthy 
exceptions are the papers by 
M.~Wilke~\cite{Wil17a}, J.~Pr\"{u}ss, G.~Simonett, and 
M.~Wilke~\cite{PSW19a}, H.~Abels, M.~Rauchecker, and M.~Wilke~\cite{ARW19a}, 
and M.~Rauchecker~\cite{Rau20a} who study various important parabolic free 
boundary problems, presupposing that the membrane makes a ninety degree 
boundary contact. In addition, in all of them, except for~\cite{ARW19a}, 
a~capillary (i.e.,~cylindrical) geometry is being studied. 
The same 
ninety degree condition is employed by H.~Garcke and 
M.~Rauchecker~\cite{GaR19a} who carry out a linearized stability 
computation at a stationary solution of a Mullins--Sekerka flow in~a 
two-dimensional bounded domain. 

\par 
The assumption of the ninety degree contact considerably simplifies the 
analysis since it allows to use reflection arguments. This does not apply in 
the case of general transversal intersection. 

\par 
The only paper, we are aware of, in which a general contact angle is being 
considered is the one by 
Ph.~Lauren\c{c}ot and Ch.~Walker~\cite{LaW20a}. These authors establish the 
unique solvability in the strong \hbox{$L_2$~sense} of a two-dimensional 
stationary transmission problem taking advantage of a particularly 
favorable geometric setting. 

\par 
Elliptic problems with boundary and transmission conditions have also been 
investigated in a series of papers by V.~Nistor and coworkers 
\cite{BMNZ10a}, \cite{LMN10a}, \cite{LNQ13a}, and~\cite{MaNi10a}. The 
motivation for these works stems from the desire to get optimal convergence 
rates for approximations used for numerical computations. Although these 
authors employ weighted \hbox{$L_2$~Sobolev} spaces, their methods 
and results are quite different from the ones presented here. 

\par 
In this paper we establish the maximal regularity of linear 
inhomogeneous parabolic 
transmission boundary value problems for the case where the interface 
intersects the boundary transversally. This is achieved by allowing the 
equations to degenerate near the intersection manifold and working in 
suitable weighted Sobolev spaces. We restrict ourselves to the simplest 
case of a fixed membrane and a single reaction-diffusion equation. 

\par 
In a forthcoming publication we shall use our present result to establish the 
local well-posedness of quasilinear equations with nonlinear boundary and 
transmission conditions. 

\par 
The author is deeply grateful to G.~Simonett for carefully reading the first 
draft of this paper, valuable suggestions, and pointing out misprints, 
errors, and the above references to related moving boundary problems. 
\section{The Main Result}\label{sec-R} 
Now we outline---in a slightly sketchy way---the 
main result of this paper. Precise definitions of notions, facts, and 
function spaces which we use here without further explanation, are given in 
the subsequent sections. 

\par 
Let $\Om$ be a bounded domain in~$\BR^m$, 
\,\hb{m\geq2}, with a smooth boundary~$\Ga$ lying locally on one side 
of~$\Om$. By a membrane in~$\oO$ we mean a smooth oriented 
hypersurface~$S$ of (the manifold)~$\oO$ with~a (possibly empty) 
boundary~$\Sa$ such that 
\hb{S\cap\Ga=\Sa}. Thus $S$~lies in~$\Om$ if 
\hb{\Sa=\es}. Otherwise, $\Sa$ is an 
\hb{(m-2)}-dimensional oriented smooth submanifold of~$\Ga$. In this case 
it is assumed that $S$ and~$\Ga$ intersect transversally. Note that 
we do not require that $S$~be connected. Hence, even if 
\hb{\Sa\neq\es}, there may exist interior membranes. However, the focus 
in this paper is on membranes with boundary. Thus we assume until further 
notice that 
\hb{\Sa\neq\es}. 

\par 
We denote by~$\nu$ the inner (unit) normal (vector field) on~$\Ga$ and 
by~$\nu_S$ the positive normal on~$S$. (Thus 
\hb{\nu_S(x)\in T_x\oO=T_x\BR^m=\{x\}\times\BR^m}, the latter being also 
identified with 
\hb{x+\BR^m\is\BR^m} for 
\hb{x\in S}.) As usual, 
\hb{\eesdot=\eesdot_S} is the jump across~$S$. We fix any 
\hb{T\in(0,\iy)} and set 
\hb{J=J_T:=[0,T]}. 

\par 
Of concern in this paper are linear reaction-diffusion equations with 
nonhomogeneous boundary and transmission conditions of the following form. 

\par 
Set 
$$ 
\bal 
\cA u 
&:=-\tdiv(a\grad u),\ \cB u:=\ga a\pl_\nu u,\cr 
\cC^0 u 
&:=\eea{u},\ \cC^1u:=\eea{a\pl_{\nu_S}u},\ \cC=(\cC^0,\cC^1), 
\eal  
\npbd  
$$ 
with $\ga$~being the trace operator on~$\Ga$. We assume (for the moment) 
that 
\hb{a\in\bar C^1\bigl((\oO\ssm S)\times J\bigr)} and 
\hb{a>0}. A~bar over a symbol for a standard function space means that its 
elements may undergo jumps across~$S$. (The usual definitions based on 
decompositions of 
\hb{\oO\ssm S} in `inner' and `outer' domains cannot be used since 
\hb{\oO\ssm S} may be connected.) Then the problem under investigation reads: 
\begin{equation}\label{R.RD} 
\bal 
\pl_tu+\cA u 
&=f 
 &&\text{ on }(\oO\ssm S)\times J,\cr 
\cB u 
&=\vp 
 &&\text{ on }(\Ga\ssm\Sa)\times J,\cr 
\cC u 
&=\psi 
 &&\text{ on }(S\ssm\Sa)\times J,\cr 
\ga_0u 
&=u_0 
 &&\text{ on }(\oO\ssm S)\times\{0\},\cr 
\eal 
\npbd 
\end{equation}  
where $\ga_0$~is the trace operator at 
\hb{t=0}.\po  

\par 
We are interested in the strong \hbox{$L_p$~solvability} of \eqref{R.RD}, 
that is, in solutions possessing second order space derivatives in~$L_p$. 
However, 
since $S$~intersects~$\Ga$, we cannot hope to get solutions which possess 
this regularity up to~$\Sa$. Instead, it is to be expected that the 
derivatives of~$u$ blow up as we approach~$\Sa$. For this reason 
we set up our problem in weighted Sobolev spaces where the weights 
control the behavior of~$\pa u$ for 
\hb{0\leq|\al|\leq2} in relation to the distance from~$\Sa$. This requires 
that the differential operator is adapted to such a setting, which means 
that the adapted `diffusion coefficient' tends to zero near~$\Sa$. 
In other words: we will have to deal with parabolic problems which 
degenerate near~$\Sa$. To describe the situation precisely, we introduce 
curvilinear coordinates near~$\Sa$ as follows. 

\par 
Since $\Sa$~is an oriented hypersurface in~$\Ga$, there exists a unique 
positive normal vector field~$\mu$ on~$\Sa$ in~$\Ga$. Given 
\hb{\sa\in\Sa}, we write~$\mu(\cdot,\sa)$ for the unique geodesic in~$\Ga$ 
satisfying 
\hb{\mu(0,\sa)=\sa} and 
\hb{\dot\mu(0,\sa)=\mu(\sa)}. Similarly, for each 
\hb{y\in\Ga} we set 
\hb{\nu(\xi,y):=y+\xi\nu(y)} for 
\hb{\xi\geq0}. Then we can choose 
\hb{\ve\in(0,1)} and a neighborhood~$\wt{U}(\ve)$ of~$\Sa$ in~$\oO$ 
with the following properties: for each 
\hb{x\in\wt{U}(\ve)} there exists  a unique triple 
$$ 
(\xi,\eta,\sa)\in N(\ve)\times\Sa 
\qa N(\ve):=[0,\ve)\times(-\ve,\ve), 
$$ 
such that 
\begin{equation}\label{R.x} 
x=x(\xi,\eta,\sa):=\nu\bigl(\xi,\mu(\eta,\sa)\bigr). 
\npbd 
\end{equation}  
Thus 
\hb{x\in\Ga\cap\wt{U}(\ve)} iff 
\hb{(\xi,\eta,\sa)\in\{0\}\times(-\ve,\ve)\times\Sa}.\po  

\par 
Now we define curvilinear derivatives for 
\hb{u\in C^2\bigl(\wt{U}(\ve)\bigr)} by 
\begin{equation}\label{R.nu} 
\pl_\nu u(x)=\pl_1(u\circ x)(\xi,\eta,\sa) 
\qb \pl_\mu u(x):=\pl_2(u\circ x)(\xi,\eta,\sa) 
\npbd 
\end{equation}  
for 
\hb{x\in\wt{U}(\ve)}. It follows that\footnote{If  
\hb{m=2}, then the last term must be disregarded. It is understood that 
similar interpretations and adaptions are to be made throughout this paper.} 
\begin{equation}\label{R.Auv} 
\cA u 
=-\bigl(\pl_\nu(a\pl_\nu u)+\pl_\mu(a\pl_\mu u) 
+\tdiv_\Sa(a\grad_\Sa u)\bigr) 
\end{equation}  
on~$\wt{U}(\ve)$, 
where $\tdiv_\Sa$ and~$\grad_\Sa$ denote the divergence and the gradient, 
respectively, in~$\Sa$ (with respect to the Riemannian metric~$g_\Sa$ 
induced by the one of~$\Ga$ which, in turn, is induced by the 
Euclidean metric on~$\oO$). 

\par 
For~$x$ given by~\eqref{R.x}, we set 
\begin{equation}\label{R.r} 
r(x):=\sqrt{\xi^2+\eta^2} 
\qa (\xi,\eta)\in N(\ve), 
\end{equation}  
which is the geodesic distance in~$\oO$ from~$x$ to~$\Sa$ 
(and not, in general, the distance in the ambient space~$\BR^m$). We fix 
\hb{\om\in C^\iy\bigl(N(\ve),[0,1]\bigr)}, depending only on~$r$, such that 
\hb{\om\sn N(\ve/3)=1} and 
\hb{\supp(\om)\is N(2\ve/3)} and set 
\begin{equation}\label{R.rh} 
\rho:=1-\om+r\om. 
\end{equation}  
Then we define on 
\begin{equation}\label{R.U} 
U:=U(\ve):=\wt{U}(\ve)\ssm\Sa 
\end{equation}  
a~singular linear reaction-diffusion operator~$\cA_U$ by 
\begin{equation}\label{R.AU} 
\cA_Uu:=-\rho^2\bigl(\pl_\nu(a\pl_\nu u)+\pl_\mu(a\pl_\mu u)\bigr) 
-\tdiv_\Sa(a\grad_\Sa u) 
\end{equation}  
for 
\hb{u\in\bar C^2(U\ssm S)}. The corresponding singular boundary 
operator is given by 
\begin{equation}\label{R.BU} 
\cB_Uu:=\ga a\rho\pl_\nu u. 
\end{equation}  

\par 
Since $S$~intersects~$\Ga$ transversally, it follows that there exists 
a smooth function 
\hb{s\sco[0,\ve)\times\Sa\ra(-\ve,\ve)} such that 
\hb{s(0,\sa)=0} for 
\hb{\sa\in\Sa} and 
\begin{equation}\label{R.xS} 
x\in\wt{U}(\ve)\cap S 
\quad\text{iff}\quad 
x=\bigl(\xi,s(\xi,\sa),\sa\bigr) 
\qa (\xi,\sa)\in[0,\ve)\times\Sa. 
\end{equation} 
Using this we associate  with~$\cA_U$ a~transmission 
operator~$\cC_U$ on~$U$ by setting 
$$ 
\bal 
\cC_U^0u 
&:=\eea{u}_{U\cap S},\cr  
\cC_U^1u 
&:=\beea{a\bigl(\nu_S^1\pl_\nu u+\nu_S^2\pl_\mu u 
+\nu_S^3(\grad_\Sa u\sn\grad_\Sa s)\bigr)_\Sa}_{U\cap S} 
\eal 
$$ 
for 
\hb{u\in\bar C^2(U\ssm S)}, where 
\hb{\prsn_\Sa=g_\Sa} and 
$$ 
(\nu_S^1,\nu_S^2,\nu_S^3) 
:=(\pl_\nu s,-1,1)\Big/\sqrt{1+(\pl_\nu s)^2+|\grad_\Sa s|_\Sa^2}.   
$$ 
Now we define a~\emph{singular transmission boundary value problem on} 
\hb{\oO\ssm S}~by putting 
\hb{V:=\oO\,\big\backslash\wt{U}(2\ve/3)} and 
$$ 
(\cA_r,\cB_r,\cC_r):=  
\left\{ 
\bal 
{}
&(\cA,\cB,\cC)
 &&\ \text{on }V,\cr 
&(\cA_U,\cB_U,\cC_U)  
 &&\ \text{on }U. 
\eal 
\right. 
\npbd  
$$ 
It follows from \eqref{R.Auv} and the properties of~$\rho$ that this 
definition is unambiguous.\po   

\par 
To introduce weighted Sobolev spaces on 
\hb{U\ssm S} we put 
\begin{equation}\label{R.us} 
\bal 
{}
\dl u\dr^2:=|u|^2 
&+|r\pl_\nu u|^2 
 +|r\pl_\mu u|^2\cr  
&+|(r\pl_\nu)^2u|^2 
 +|r\pl_\nu(r\pl_\mu u)|^2
 +|r\pl_\mu(r\pl_\nu u)|^2 
 +|(r\pl_\mu)^2u|^2\cr
&+|\naSa u|^2+|\naSa^2u|^2, 
\eal 
\end{equation}  
where $\naSa$~is the Levi--Civita connection on~$\Sa$ 
for the metric~$g_\Sa$. Moreover, 
\hb{1<p<\iy} and 
\begin{equation}\label{R.WU} 
\|u\|_{\bar W_{\coW p}^2(U\setminus S;r)} 
:=\Bigl(\int_{U\setminus S}\dl u\dr^p\,\frac{d(\xi,\eta)}{r^2} 
\,d\vol_\Sa\Bigr)^{1/p}. 
\end{equation}  
Then 
\hb{\bar W_{\coW p}^2(U\ssm S;r)} is the completion of 
\hb{\bar C^2(U\ssm S)} in 
\hb{L_{1,\loc}(U\ssm S)} with respect to the norm~%
\hb{\Vsdot_{\bar W_{\coW p}^2(U\setminus S;r)}}. 

\par 
The (global) weighted Sobolev space 
$$ 
\cX_p^2:=\bar W_{\coW p}^2(\oO\ssm S;r) 
$$ 
consists of all 
\hb{u\in L_{1,\loc}(\oO\ssm S)} with 
\hb{u\bsn U\in\bar W_{\coW p}^2(U\ssm S;r)} and 
\hb{u\bsn V\in\bar W_{\coW p}^2(V\ssm S)}. It is a Banach space 
with the norm 
$$ 
u\mt\big\|u\sn U\big\|_{\bar W_{\coW p}^2(U\setminus S;r)} 
+\big\|u\sn V\big\|_{\bar W_{\coW p}^2(V\setminus S)}, 
$$ 
whose topology is independent of the specific choice of $\ve$ and~$\om$. 
Similarly, the Lebesgue space 
$$ 
\cX_p^0:=\bar W_{\coW p}^0(\Om\ssm S;r) 
$$ 
is obtained by replacing~$\dl u\dr$ in \eqref{R.WU} by~$|u|$. Moreover, 
\begin{equation}\label{R.int} 
\cX_p^{2-2/p}:=
\bar W_{\coW p}^{2-2/p}(\oO\ssm S;r) 
:=(\cX_p^0,\cX_p^2)_{1-1/p,p},  
\npbd 
\end{equation}  
where 
\hb{\prsn_{\ta,p}} is the real interpolation functor of exponent~$\ta$.\po  

\par 
We also need time-dependent anisotropic spaces. For this we use the notation 
\hb{s/\mf2:=(s,\,s/2)}, 
\ \hb{0\leq s\leq2}. Then 
$$ 
\cX_p^{2/\mf2} 
:=\bar W_{\coW p}^{2/\mf2}\bigl((\oO\ssm S)\times J;r\bigr) 
:=L_p(J,\cX_p^2)\cap W_{\coW p}^1(J,\cX_p^0)
$$  
and 
\hb{\cX_p^{0/\mf2}:=L_p(J,\cX_p^0)}. If 
\hb{X\in\{\Ga,S\}} and 
\hb{s\in\{1-1/p,\ 2-1/p\}}, then 
$$ 
\bar W_{\coW p}^{s/\mf2}\bigl((X\ssm\Sa)\times J;r\bigr) 
:=L_p\bigl(J,\bar W_{\coW p}^s(X\ssm\Sa;r)\bigr) 
\cap W_{\coW p}^{s/2}\bigl(J,L_p(X\ssm\Sa;r)\bigr). 
$$ 
Here the 
\hb{\bar W_{\coW p}^s(X\ssm\Sa;r)} are trace spaces of~$\cX_p^2$
(cf.~\eqref{MB.Ga} and \eqref{MB.S}). Moreover, 
$$ 
\bal 
\cY_p 
 :=\bar W_{\coW p}^{(1-1/p)/\mf2}\bigl((\Ga\ssm\Sa)\times J;r\bigr) 
&\oplus\bar W_{\coW p}^{(2-1/p)/\mf2}\bigl((S\ssm\Sa)\times J;r\bigr)\cr  
&\oplus\bar W_{\coW p}^{(1-1/p)/\mf2}\bigl((S\ssm\Sa)\times J;r\bigr).  
\eal 
$$ 

\par 
By 
\hb{\bar{BC}(\oO\ssm S)} we mean the space of bounded and continuous 
functions (with possible jumps across~$S$), endowed with the maximum norm. 
Then 
\hb{\bar{BC}^1(\oO\ssm S;r)} is the Banach space of all 
\hb{u\in\bar{BC}(\oO\ssm S)} with 
\hb{\pl_ju\in\bar{BC}(V\ssm S)}, 
\ \hb{1\leq j\leq m}, and 
$$ 
\rho\pl_\nu u,\rho\pl_\mu u\in\bar{BC}(U\ssm S) 
\qb u\sn\Sa\in BC^1(\Sa). 
$$ 
Furthermore, 
$$ 
\bar{BC}^{1/\mf2}\bigl((\oO\ssm S)\times J;r\bigr) 
:=C\bigl(J,\bar{BC}^1(\oO\ssm S;r)\bigr) 
\cap C^{1/2}\bigl(J,\bar{BC}(\oO\ssm S)\bigr). 
$$ 
To indicate the nonautonomous structure of \eqref{R.RD}, we write 
\hb{a(t):=a(\cdot,t)} and, correspondingly, $\cA(t)$, $\cB(t)$, and~$\cC(t)$. 

\par 
Now we are ready to formulate the main result of this paper, the optimal 
solvability of linear reaction-diffusion transmission 
boundary value problems. 
\begin{theorem}\label{thm-R.RD} 
Let 
\hb{1<p<\iy} with 
\hb{p\notin\{3/2,\,3\}} and 
$$ 
a\in\bar{BC}^{1/\mf2}\bigl((\oO\ssm S)\times J;r\bigr) 
\qa a\geq\ul{\al}, 
$$ 
for some 
\hb{\ul{\al}\in(0,1)}. Suppose 
$$ 
\bigl(f,(\vp,\psi^0,\psi^1),u_0\bigr) 
\in\cX_p^{0/\mf2}\oplus\cY_p\oplus\cX_p^{2-2/p} 
$$ 
and that the following compatibility conditions are satisfied:
$$ 
\bal 
{\rm(i)}\quad 
\cC_r^0(0)u_0 
&=\psi^0(0), 
&\quad\text{if\/ }&&3/2<p<3,\cr 
{\rm(ii)}\quad 
\cB_r(0)u_0 
&=\vp(0),  
 \ \cC_r(0)u_0=\psi(0), 
&\quad\text{if\/ }&&p>3, 
\eal 
$$ 
where 
\hb{\psi:=(\psi^0,\psi^1)}. Then 
\begin{equation}\label{R.P} 
\bal 
\bal 
\pl_tu+\cA_r(t)u    &=f      &&\text{ on }(\oO\ssm S)\times J,\cr 
\cB_r(t)u           &=\vp    &&\text{ on }(\Ga\ssm\Sa)\times J,\cr 
\cC_r(t)u           &=\psi   &&\text{ on }(S\times\Sa)\times J,\cr 
\ga_0u              &=u_0    &&\text{ on }(\oO\ssm S)\times\{0\}  
\eal 
\eal 
\npbd 
\end{equation} 
has a unique solution 
\hb{u\in\cX_p^{2/\mf2}}. It depends continuously on the data.\po  
\end{theorem} 
\setcounter{corollary}{1} 
\begin{corollary}\label{cor-R.RD} 
Suppose $a$~is independent of~$t$, that is, 
\hb{a\in\bar{BC}^1(\oO\ssm S;r)}. Set 
$$ 
\cX_{p,0}^2:=  
\left\{ 
\bal 
{}
&\ \cX_p^2,  
 &\quad 1  &<p<3/2,\cr 
&\{\,u\in\cX_p^2\ ;\ \cC^0u=0\,\}, 
 &\quad 3/2&<p<3,\cr 
&\{\,u\in\cX_p^2\ ;\ \cB u=0,\ \cC u=0\,\}, 
 &\quad 3  &<p<\iy,
\eal 
\right. 
$$ 
and 
\hb{A_r:=\cA_r\sn\cX_{p,0}^2}. Then~$-A_r$, considered as a 
linear operator in~$\cX_p^0$ with domain~$\cX_{p,0}^2$, generates 
on~$\cX_p^0$ a strongly continuous analytic semigroup.  
\end{corollary} 
\begin{proof} 
The theorem implies that $A_r$~has the property of maximal 
\hbox{$\cX_p^0$~regularity}. This fact is well-known to imply the claim 
(e.g.,~\cite[Capter~III]{Ama95a} or~\cite{DHP07a}).  
\end{proof} 
Theorem~\ref{thm-R.RD} is a special case of the much more general 
Theorems \ref{thm-L.MR} and \ref{thm-MB.P}. They also include Dirichlet 
boundary conditions and apply to transmission problems in general 
Riemannian manifolds with boundary and bounded geometry.  

\par 
The situation is considerably simpler if 
\hb{\Sa=\es}, that is, if only interior transmission hypersurfaces 
are present. Of course, if 
\hb{S=\es}, then \eqref{R.P} reduces to a linear reaction-diffusion 
equation with inhomogeneous boundary conditions. In these cases no 
degenerations do occur. 

\par 
We refrain from considering operators~$\ArBr$ with lower 
order terms. This case will be covered by the forthcoming 
quasilinear result. 

\par 
In the case of an interior transmission surface (that is, 
\hb{\Sa=\es}) and if $a$~is independent of~$t$, Theorem~\ref{thm-R.RD} 
is a special case of Theorem~6.5.1 in~\cite{PrS16a}. The latter 
theorem applies to systems and provides an 
\hbox{$L_p$-$L_q$} theory. 

\par 
If 
\hb{\Sa\neq\es}, then the basic difficulty in proving Theorem~\ref{thm-R.RD} 
stems from the fact that 
\hb{\oO\ssm\Sa} and, consequently, 
\hb{S\ssm\Sa} and 
\hb{\Ga\ssm\Sa}, are no longer compact. The fundamental observation which 
makes the proofs work is the fact that we can consider 
\hb{\oO\ssm\Sa} as a (noncompact) Riemannian manifold with a metric~$g$ 
which coincides on 
\hb{U(\ve/3)} with the singular metric 
\hb{r^{-2}d\nu\otimes d\mu+g_\Sa} and on~$V$ with the Euclidean metric. 
With respect to this metric, $\cA_r$~is a uniformly elliptic operator. 

\par 
Theorems \ref{thm-H.C}  and~\ref{thm-S.M} show that 
\hb{(\oO\ssm\Sa,\,g)} is a uniformly regular Riemannian manifold in the sense 
of~\cite{Ama12b}. Thus we are led to consider linear parabolic equations with 
boundary and transmission conditions on such manifolds. As in the compact 
case, by means of local coordinates the problem is reduced to Euclidean 
settings. However, since we have to deal with noncompact manifolds, we have 
to handle simultaneously infinitely many model problems. In order for this 
technique to work, we have to establish uniform estimates which are in a 
suitable sense independent of the specific local coordinates. In addition, 
special care has to be taken in `gluing together the local model problems'. 
These are no points to worry about in the compact case. 

\par
In our earlier paper~\cite{Ama17a} we have established an optimal existence 
theory for linear parabolic equations on uniformly regular Riemannian 
manifolds without boundary. The present proof extends those arguments to 
the case of manifolds with boundary. The presence of boundary and 
transmission conditions adds considerably to the complexity of the problem 
and makes the paper rather heavy. 

\par 
In Section~\ref{sec-U} we collect the needed background information. 
In the subsequent two sections we establish the differential geometric 
foundation of transmission surfaces in uniformly regular and singular 
Riemannian manifolds.

\par 
After having introduced the relevant function spaces in 
Section~\ref{sec-F}, we present in Section~\ref{sec-L} the basic maximal 
regularity theorem in anisotropic Sobolev spaces for linear 
non-autonomous reaction-diffusion equations with nonhomogeneous boundary and 
transmission conditions on general uniformly regular Riemannian manifolds. 
Its rather complex proof occupies the next 
five sections. Finally, in the last section it is shown that our general 
results apply to the Euclidean setting presented here. 
\section{Uniformly Regular Riemannian Manifolds}\label{sec-U} 
In this section we recall the definition of uniformly regular Riemannian 
manifolds and collect those 
properties of which we will make use. Details can be found in 
\cite{Ama12c}, \cite{Ama12b}, \cite{Ama15a}, and in the comprehensive 
presentation~\cite{AmaVolIII21a}. Thus we shall be rather brief. 

\par 
We use standard notation from differential geometry and function space 
theory. In particular, an upper, resp.~lower, asterisk on a symbol for a 
diffeomorphism denominates the corresponding pull-back, resp.\ push-forward 
(of tensors). 
By~$c$, resp.~$c(\al)$ etc., we denote constants~%
\hb{\geq1} which can vary from occurrence  to occurrence. 
Assume $S$~is a nonempty set. On the cone of nonnegative functions on~$S$ 
we define an equivalence relation~%
\hb{{}\sim{}} by 
\hb{f\sim g} iff 
\hb{f(s)/c\leq g(s)\leq cf(s)}, 
\ \hb{s\in S}. 

\par
An \hbox{$m$-dimensional} manifold is a separable metrizable space equipped 
with an \hbox{$m$-dimensional} smooth structure. We always work in the 
smooth category. 

\par 
Let $M$ be an \hbox{$m$-dimensional} manifold with or without boundary. If 
$\ka$~is a local chart, then we use~$U_{\coU\ka}$ for its domain, the 
coordinate patch associated with~$\ka$. The chart is normalized if 
\hb{\ka(U_{\coU\ka})=Q_\ka^m}, where 
\hb{Q_\ka^m=(-1,1)^m} if 
\hb{U\is\ci M}, the interior of~$M$, and 
\hb{Q_\ka^m=[0,1)\times(-1,1)^{m-1}} otherwise. An atlas~$\gK$ is normalized 
if it consists of normalized charts. It is shrinkable if it normalized and 
there exists 
\hb{r\in(0,1)} such that 
\hb{\bigl\{\,\ka^{-1}(rQ_\ka^m)\ ;\ \ka\in\gK\,\bigr\}} is a covering 
of~$M$. It has finite multiplicity if there exists 
\hb{k\in\BN} such that any intersection of more than~$k$ coordinate 
patches is empty. 

\bigskip 
The atlas~$\gK$ is \emph{uniformly regular}~(ur) if 
\begin{equation}\label{U.K} 
\bal 
{\rm(i)}\quad 
&\text{it is shrinkable and has finite multiplicity;}\cr 
{\rm(ii)}\quad 
&\wk\circ\ka^{-1}\in\BUC^\iy\bigl(\ka(U_\kwk),\BR^m\bigr)
 \text{ and }\|\wk\circ\ka^{-1}\|_{k,\iy}\leq c(k)\cr 
\noalign{\vskip-1\jot} 
&\text{for }\ka,\wk\in\gK\text{ and }k\in\BN,
 \text{ where }U_\kwk:=U_{\coU\ka}\cap U_\wk.   
\eal 
\end{equation}  
Two ur atlases $\gK$ and~$\wt{\gK}$ are equivalent if 
\begin{equation}\label{U.e} 
\kern-8pt 
\bal 
{\rm(i)}\quad 
&\text{there exists $k\in\BN$ such that each coordinate patch of $\gK$}\cr 
\noalign{\vskip-1.5\jot} 
&\text{meets at most $k$ coordinate patches of $\wt{\gK}$, 
 and vice versa;}\cr 
{\rm(ii)}\quad 
&\text{condition~\eqref{U.K}(ii) holds for all } 
 (\ka,\wt{\ka})\text{ and }(\wt{\ka},\ka)\text{ belonging to } 
 \gK\times\wt{\gK}.   
\eal 
\kern-10pt 
\end{equation}  
This defines an equivalence relation in the class of all ur~atlases. An 
equivalence class thereof is a ur~structure. By~a \emph{ur~manifold} 
we mean a manifold equipped with a ur~structure. 
A~Riemannian metric~$g$ on a~ur manifold~$M$ is~\emph{ur} if, 
given a~ur atlas~$\gK$,
\begin{equation}\label{U.g} 
\bal 
{\rm(i)}\quad 
&\ka_*g\sim g_m,\ \ka\in\gK;\cr 
{\rm(ii)}\quad 
&\|\ka_*g\|_{k,\iy}\leq c(k),\ \ka\in\gK,\ k\in\BN.    
\eal 
\end{equation}  
Here 
\hb{g_m:=\prsn=dx^2} is the Euclidean 
metric\footnote{We use the same symbol for a Riemannian metric 
and its restrictions to sub\-manifolds of the same dimension.} 
on~$\BR^m$ and (i)~is understood 
in the sense of quadratic forms. This concept is well-defined, 
independently of the specific~$\gK$. A~\emph{uniformly regular Riemannian} 
(urR) \emph{manifold} is~a ur~manifold, endowed with~a urR metric. 
\setcounter{remarks}{0} 
\begin{remarks}\label{rem-U.R}  
(a) 
Given a (nonempty) 
subset~$S$ of~$M$ and an atlas~$\gK$, 
$$ 
\gK_S:=\{\,\ka\in\gK\ ;\ U_{\coU\ka}\cap S\neq\es\,\}\ . 
$$ 
We say that $\gK$~is normalized \emph{on}~$S$, resp.\ has finite 
multiplicity \emph{on}~$S$, resp.\ is shrinkable \emph{on}~$S$ 
if $\gK_S$~possesses the respective properties. Moreover, $\gK$~is~ur 
\emph{on}~$S$ if \eqref{U.K} applies with~$\gK$ replaced 
by~$\gK_S$. Similarly, two atlases $\gK$ and~$\wt{\gK}$, which are 
ur \emph{on}~$S$, are equivalent \emph{on}~$S$ if \eqref{U.e} holds with 
$\gK$ and~$\wt{\gK}$ replaced by $\gK_S$  and~$\wt{\gK}_S$, respectively. 
This induces a \hbox{ur~structure} \emph{on}~$S$. Finally, $M$~is~ur 
\emph{on}~$S$ if it is equipped with~a \hbox{ur~structure} on~$S$. 

\par 
(b) 
Suppose $\gK$~is a ur atlas for~$M$ on~$S$. Given any 
\hb{\ve>0}, there exists a ur atlas~$\gK'$ on~$S$ such that 
\hb{\diam_g(U_{\coU\ka})<\ve} for 
\hb{\ka\in\gK'}, where $\diam_g$~is the diameter with respect to the 
Riemannian distance~$d_g$.\qed 
\end{remarks} 
In the following examples we use the natural ur~structure 
(e.g.,~the product ur~structure in Example~\ref{exa-U.ex}(c)) 
if nothing else is mentioned.  
\setcounter{examples}{1}
\begin{examples}\label{exa-U.ex} 
(a) 
Each compact Riemannian manifold is~a urR manifold and its ur~structure is 
unique. 

\par 
(b) 
Let $\Om$ be a bounded domain in~$\BR^m$ with a smooth boundary such that 
$\Om$~lies locally on one side of it. Then $(\oO,g_m)$ is~a 
urR manifold. 

\par 
More generally, suppose that $\Om$~is an unbounded open subset of~$\BR^m$ 
whose boundary is ur in the sense of F.E.~Browder~\cite{Bro59a} (also see 
\cite[IV.\S4]{LSU68a}). Then $\oOgm$~is a urR manifold. 
In particular, $\Rmgm$ and~$\Hmgm$ are urR manifolds, where 
\hb{\BH^m:=\BR_+\times\BR^{m-1}} is the closed right half-space in~$\BR^m$. 

\par 
(c) 
If $\Migi$, 
\ \hb{i=1,2}, are urR manifolds and at most one of them has a nonempty 
boundary, then 
\hb{(M_1\times M_2,\ g_1+g_2)} is~a urR manifold. 

(d) 
Assume $M$~is a manifold and $N$~a topological space. Let 
\hb{f\colon N\ra M} be a homeomorphism. If $\gK$ is an atlas for~$M$, then 
\hb{f^*\gK:=\{\,f^*\ka\ ;\ \ka\in\gK\,\}} is an atlas for~$N$ which induces 
the smooth `pull-back' structure on~$N$. If $\gK$~is ur, 
then $f^*\gK$ also is~ur. 

\par 
Let $\Mg$ be~a urR manifold. Then 
\hb{f^*\Mg:=(N,f^*g)} is~a urR manifold and the map 
\hb{f\colon(N,f^*g)\ra\Mg} is an isometric diffeomorphism.\qed 
\end{examples} 
It follows from these examples, for instance, that the cylinders 
\hb{\BR\times M_1} or 
\hb{\BR_+\times M_2}, where $M_i$~are compact Riemannian manifolds with  
\hb{\pl M_2=\es}, are 
urR manifolds. More generally, Riemannian manifolds with cylindrical 
ends are urR manifolds (see~\cite{Ama15a} where more examples 
are discussed). 

\par 
Without going into detail, we mention that a Riemannian manifold without 
boundary is~a urR manifold iff it has bounded geometry (see~\cite{Ama12b} 
for one half of this assertion and \cite{DSS16a} for the other half). Thus, 
for example, $(\ci\BH^m,g_m)$~is \emph{not} a urR manifold. 
A~Riemannian manifold with boundary is~a urR manifold iff it has bounded 
geometry in the sense of Th.~Schick~\cite{Schi01a} (also see 
\cite{AGN19a}, \cite{AGN19b}, \cite{AGN19c}, \cite{GSchn13a} for related 
definitions). Detailed proofs of these equivalences can be found 
in~\cite{AmaVolIII21a}. 
\section{Uniformly Regular Hypersurfaces}\label{sec-H} 
Let $\Mg$ be an oriented Riemannian manifold with (possibly empty) 
boundary~$\Ga$. If it is not empty, then there exists 
a unique \emph{inner} (unit) \emph{normal} vector field 
\hb{\nu=\nu_\Ga} on~$\Ga$, that is, a~smooth section of~$T_\Ga M$, 
the restriction 
of the tangent bundle $TM$ of~$M$ to~$\Ga$. Furthermore, $\Ga$~is oriented 
by the inner normal in the usual sense. 

\par 
Suppose that $S$~is an oriented hypersurface in~$\ci M$, an 
embedded submanifold of codimension~$1$. Then there is a unique 
positive (unit) normal vector field~$\nu_S$ on~$S$, 
where `positive' means that 
\hb{\bigl[\nu_S(p),\ba_1,\ldots,\ba_{m-1}\bigr]} is a positive basis 
for~$T_pM$ if 
\hb{[\ba_1,\ldots,\ba_{m-1}]} is one for~$T_pS$. 

\par 
Let 
\hb{Z\in\{\Ga,S\}}. Then we write 
$$ 
\ga_p^Z(t):=\exp_p\bigl(t\nu_Z(p)\bigr) 
\qa t\in I_Z\bigl(\ve(p)\bigr).  
$$ 
This means that, given 
\hb{p\in Z}, 
\ $\ga_p^Z$~is the unique geodesic in~$M$ satisfying 
\hb{\ga_p^Z(0)=p} and 
\hb{\dot\ga_p^Z(0)=\nu_Z(p)} and being defined (at least) 
on~$I_Z\bigl(\ve(p)\bigr)$, where 
\hb{I_\Ga\bigl(\ve(p)\bigr)=\bigl[0,\ve(p)\bigr)} and 
\hb{I_S\bigl(\ve(p)\bigr)=\bigl(-\ve(p),\ve(p)\bigr)} for some
\hb{\ve(p)>0}. Note that 
\hb{\ga_p^\Ga(t)\in\ci M} for 
\hb{t>0}. 

\par 
We say that $Z$~has a \emph{uniform} normal geodesic \emph{tubular 
neighborhood} of width~$\ve$ if the following is true: there exist 
\hb{\ve>0} and an open neighborhood~$Z(\ve)$ of~$Z$ in~$M$ 
such that 
\begin{equation}\label{H.pN} 
\vp_Z\sco Z(\ve)\ra I_Z(\ve)\times Z 
\quad\text{with}\quad 
\vp_Z^{-1}(t,p)=\ga_p^Z(t) 
\end{equation}  
is a diffeomorphism satisfying 
\hb{\vp_Z(Z)=\{0\}\times Z}. If 
\hb{Z=\Ga}, then a uniform tubular neighborhood is~a uniform \emph{collar}. 

\par 
Given any embedded submanifold~$C$ of~$M$, with or without boundary, 
we denote by~$g_C$ the pull-back metric~$\ia^*g$, where 
\hb{\ia\sco C\hr M} is the natural embedding. 

\par 
Now we suppose that 
\begin{equation}\label{H.M} 
\Mg\text{ is an $m$-dimensional oriented urR manifold}. 
\npbd 
\end{equation}  
This means that there exists an oriented ur atlas for~$M$.\po  

\par 
Let $S$ be a hypersurface with boundary~$\Sa$ such that 
\hb{\Sa=S\cap\Ga}. Thus 
\hb{S\is\ci M} if 
\hb{\Sa=\es}. An atlas~$\gK$ for~$M$ is \hbox{$S$\emph{-adapted}} 
if for each 
\hb{\ka\in\gK_S} one of the following alternatives applies: 
$$ 
\bal 
{\rm(i)}\quad 
&\ka\notin\gK_\Ga.\text{ Then }Q_\ka^m=(-1,1)^m\text{ and}\cr  
&\ka(S\cap U_{\coU\ka})=\{0\}\times(-1,1)^{m-1};\cr 
{\rm(ii)}\quad 
&\ka\in\gK_\Ga. 
 \text{ Then }Q_\ka^m=[0,1)\times(-1,1)^{m-1},\cr 
&\ka(\Ga\cap U_{\coU\ka})=\{0\}\times(-1,1)^{m-1},\text{ and}\cr 
&\ka(\Sa\cap U_{\coU\ka})=\{0\}^2\times(-1,1)^{m-2}. 
\eal 
$$ 
Then $S$~is~a \emph{regularly embedded hypersurface in}~$M$, 
a~\emph{membrane} for short, if there exists an oriented ur atlas~$\gK$ 
for~$M$ which is \hbox{$S$-adapted}. 

\par 
Let $S$ be a membrane. Each \hbox{$S$-adapted} atlas for~$M$ induces 
(by restriction) a~ur structure and~a (natural) orientation on~$S$. 
Moreover, the ur structure and the orientation of~$S$ are independent 
of the specific choice of~$\gK$. 

\par 
For the proof of all this and the following theorem we refer 
to~\cite{AmaVolIII21a}. 
\setcounter{theorem}{0} 
\begin{theorem}\label{thm-H.S} 
Let \eqref{H.M} be satisfied and suppose $S$~is a membrane in~$M$. Assume 
\hb{Z\in\{\Ga,S\}}. Then 
\begin{itemize} 
\item[{\rm(i)}] 
$(Z,g_Z)$~is an 
(\hb{m-1})-dimensional oriented urR manifold. 
\item[{\rm(ii)}] 
If 
\hb{\Sa=\pl S\neq\es}, then $\Sa$~is a membrane in~$\Ga$ without boundary. 
\item[{\rm(iii)}] 
Let 
\hb{\Sa=\es} if 
\hb{Z=S}. Then $Z$~has a uniform tubular neighborhood 
$$ 
\vp_Z\sco Z(\ve)\ra I_Z(\ve)\times Z 
$$  
and 
\hb{\vp_{Z*}g\sim ds^2+g_Z}.
Moreover, $\vp_Z$~is an orientation preserving diffeomorphism.   
\item[{\rm(iv)}] 
Suppose 
\hb{\Sa\neq\es}. Then, given 
\hb{\rho>0}, there exists 
\hb{\ve(\rho)>0} such that 
$$ 
S\cap\bigl\{\,q\in M\ ;\ d_g(q,\Ga)>\rho\,\bigr\} 
\npbd 
$$ 
has a uniform tubular neighborhood of width~$\ve(\rho)$ in~$\ci M$.\po  
\end{itemize} 
\end{theorem} 
Now we suppose that $S$~is a membrane with 
\hb{\Sa\neq\es}. It follows from (ii) and~(iii) that 
$\Sa$~has a uniform tubular neighborhood 
\hb{\psi\sco\Sa^\Ga(\ve)\ra(-\ve,\ve)\times\Sa} \emph{in}~$\Ga$ for some 
\hb{\ve>0}. By part~(iii), $\Ga$~has a uniform collar 
\hb{\vp\sco\Ga(\ve)\ra[0,\ve)\times\Ga} in~$M$, where we assume without 
loss of generality that $\vp$ and~$\psi$ are of the same width. Then 
$$ 
\Sa(\ve):=\bigl\{\,\ga_q^\Ga(t)\ ;\ q\in\Sa^\Ga(\ve),\ 0\leq t<\ve\,\bigr\} 
$$ 
is an open neighborhood of~$\Sa$ in~$M$ and 
$$ 
\chi\sco\Sa(\ve)\ra[0,\ve)\times(-\ve,\ve)\times\Sa 
\quad\text{with}\quad  
\chi^{-1}(x,y,\sa):=\vp^{-1}\bigl(x,\psi^{-1}(y,\sa)\bigr) 
$$ 
is an orientation preserving diffeomorphism, a~\emph{tubular 
neighborhood of~$\Sa$ in~$M$} of width~$\ve$. 

\par 
We refer once more to~\cite{AmaVolIII21a} for the proof of the next 
theorem. Henceforth, 
\hb{h:=g_\Sa} and 
\hb{N(\ve):=[0,\ve)\times(-\ve,\ve)}. 
\begin{theorem}\label{thm-H.Si} 
Assume \eqref{H.M} and $S$~is a membrane with nonempty 
boundary~$\Sa$. Then 
\begin{equation}\label{H.g} 
\chi_*g\sim dx^2+dy^2+h. 
\end{equation}  
\end{theorem} 
It follows that 
$$ 
\Sa_\sa(\ve):=\chi^{-1}\bigl(N(\ve)\times\{\sa\}\bigr) 
$$ 
is for each 
\hb{\sa\in\Sa} a~\hbox{$2$-dimensional} submanifold of~$\Sa(\ve)$ and 
\hb{S\cap\Sa_\sa(\ve)} is~a \hbox{$1$-dimensional} submanifold 
of~$\Sa_\sa(\ve)$. 
\setcounter{remark}{2}
\begin{remark}\label{rem-H.2} 
\uti{The two-dimensional case} 
Suppose 
\hb{\dim(M)=2} and 
\hb{\Sa\neq\es}. It follows from Theorem~\ref{thm-H.S}(ii) and the fact that 
$M$~has a countable base that $\Sa$~is a countable discrete subspace of~$M$. 
Thus we can find 
\hb{\ve>0} with the following properties: if we denote by 
\hb{\psi^{-1}(\cdot,\sa)\sco(-\ve,\ve)\ra\Ga} the local arc-length 
parametrization of~$\Ga$ with 
\hb{\psi^{-1}(0,\sa)=\sa} for 
\hb{\sa\in\Sa}, then the above definition of~$\chi$ applies and defines 
a~tubular neighborhood of~$\Sa$ in~$M$ of width~$\ve$. 

\par 
Note that $\Sa(\ve)$~is the countable pair-wise disjoint union 
of~$\Sa_\sa(\ve)$, 
\ \hb{\sa\in\Sa}. The term~$+h$ in \eqref{H.g} (and everywhere else) has to 
be disregarded and the volume measure of~$\Sa$ is the counting measure. 
Thus in this case integration with respect to~$d\vol_\Sa$ reduces to 
summation over 
\hb{\sa\in\Sa}.\qed 
\end{remark} 
Now we restrict the class of membranes under consideration 
by requiring that $S$~\emph{intersects~$\Ga$ uniformly transversally}. 
This means the following: there exists 
\hb{f\in C^\iy\bigl([0,\ve)\times\Sa,\,(-\ve,\ve)\bigr)} such that, setting 
\hb{f_\sa:=f(\cdot,\sa)},
\begin{equation}\label{H.f} 
\bal 
{\rm(i)}\quad 
&f_\sa(0)=0,\ \sa\in\Sa;\cr 
{\rm(ii)}\quad 
&\text{Given $\ol{\ve}\in(0,\ve)$, 
there exists $\ol{\rho}\in(0,\ve)$ with}\cr 
&\qquad\quad 
 |f_\sa(x)|\leq\ol{\rho},\ 0\leq x\leq\ol{\ve},\ \sa\in\Sa;\cr 
 {\rm(iii)}\quad 
&|\pl f_\sa(0)|\leq c,\ \sa\in\Sa;\cr 
 {\rm(iv)}\quad 
&\chi\bigl(S\cap\Sa_\sa(\ve)\bigr)=\graph(f_\sa)\times\{\sa\},\ \sa\in\Sa. 
\eal 
\end{equation}  
In general, a~submanifold~$C$ of a manifold~$B$ intersects~$\pl B$ 
transversally if 
$$ 
T_pC+T_p\pl B=T_pB 
\qa p\in\pl B. 
$$ 

\par 
The following theorem furnishes an important large class of urR 
manifolds and membranes intersecting the boundary uniformly transversally. 
\setcounter{theorem}{3}
\begin{theorem}\label{thm-H.C} Let $\Mg$ be a compact oriented Riemannian 
manifold with boundary~$\Ga$. 
Assume $S$~is an oriented hypersurface in~$M$ with nonempty boundary 
\hb{\Sa\is\Ga} and $S$~intersects~$\Ga$ transversally. Then $\Mg$~is a urR 
manifold and $S$~is a membrane intersecting~$\Ga$ uniformly transversally. 
\end{theorem} 
\begin{proof} 
Example~\ref{exa-U.ex}(a) guarantees that $\Mg$~is an oriented urR 
manifold. Thus $(\Ga,g_\Ga)$ is an oriented urR manifold by 
Theorem~\ref{thm-H.S}(i). Since $S$~intersects~$\Ga$ transversally, it is a 
well-known consequence of the implicit function theorem that $\Sa$~is a 
compact hypersurface in~$\Ga$ without boundary. It is oriented, being the 
boundary of the oriented manifold~$S$. Hence, invoking 
Example~\ref{exa-U.ex}(a) once more, $(\Sa,g_\Sa)$ is an oriented urR 
manifold. As it is compact, it has a uniform tubular neighborhood 
in~$\Ga$. Thus, $\Ga$~having a uniform collar, $\Sa$~has a uniform 
tubular neighborhood~$\chi$ in~$M$ of some width~$\ve$. 

\par 
Since $S$~intersects~$\Ga$ transversally, 
\hb{\chi\bigl(S\cap\Sa_\sa(\ve)\bigr)} can be represented as the graph 
of a smooth function 
\hb{f_\sa\sco[0,\ve)\ra(-\ve,\ve)} with 
\hb{f_\sa(0)=0}, and $f_\sa$~depends smoothly on 
\hb{\sa\in\Sa}. The compactness of~$\Sa$ implies that \eqref{H.f} is true. 
Hence $S$~intersects~$\Ga$ uniformly transversally. 
Now, due to the compactness of~$S$, it is not difficult to see that 
$S$~is a regularly embedded submanifold of~$M$. The theorem is proved.  
\end{proof} 
\setcounter{remarks}{4} 
\begin{remarks}\label{rem-H.C}  
(a) 
This theorem applies to the case 
\hb{\Mg=\oOgm} considered in Section~\ref{sec-R}. 

\par
(b) 
Suppose $\Mg$~is an oriented urR manifold and $S$~a membrane without 
boundary. Then the fact that $S$~has a uniform tubular neighborhood 
in~$\ci M$ prevents~$S$ from either 
reaching~$\Ga$ or `collapsing' at infinity.\qed  
\end{remarks} 
\section{The Singular Manifold}\label{sec-S} 
In this section 
\begin{equation}\label{S.ass} 
\bal 
{}
&\Mg\text{ is an oriented urR manifold and}\cr 
\noalign{\vskip-1\jot}
&S\text{ a membrane with nonempty boundary }\Sa\cr 
\noalign{\vskip-1\jot}&\text{such that $S$ intersects $\Ga$ uniformly 
transversally}.
\eal 
\end{equation}  
By Theorem~\ref{thm-H.S} and the considerations following it we can choose 
a uniform tubular neighborhood 
\begin{equation}\label{S.t} 
\chi\sco\Sa(\ve)\ra N(\ve)\times\Sa. 
\end{equation}  
We write  
\hb{D(\ve):=\bigl\{\,(x,y)\in\BR^2\ ;\ x^2+y^2<\ve^2,\ x\geq0\,\bigr\}}. 
Then 
\begin{equation}\label{S.W} 
\wt{U}(\ve):=\chi^{-1}\bigl(D(\ve)\times\Sa\bigr) 
\end{equation}  
is an open neighborhood of~$\Sa$ in~$M$ contained in~$\Sa(\ve)$. We put 
$$ 
\wh{M}:=M\ssm\Sa 
\qb U(\ve):=\wt{U}(\ve)\cap\wh{M}=\wt{U}(\ve)\ssm\Sa. 
$$ 
Furthermore, $r$ and~$\rho$ are given by \eqref{R.r} and \eqref{R.rh}, 
respectively. Then we define a Riemannian metric~$\wh{g}$ on~$\wh{M}$ by 
\begin{equation}\label{S.g} 
\wh{g}:=  
\left\{ 
\bal 
{}
&g
 &&\ \text{on }M\big\backslash\wt{U}(\ve),\cr 
&\chi^*\Bigl(\frac{dx^2+dy^2}{\rho^2(x,y)}+h\Bigr)
 &&\ \text{on }U(\ve). 
\eal 
\right. 
\end{equation}  
Note that, see Theorem~\ref{thm-H.Si}, 
\begin{equation}\label{S.gg} 
\wh{g}\sim g\text{ on }M\big\backslash\wt{U}(\ve/3) 
\end{equation}  
and 
$$ 
\wh{g}=\chi^*\Bigl(\frac{dx^2+dy^2}{x^2+y^2}+h\Bigr) 
\text{ on }U(\ve/3). 
\npbd 
$$ 
Hence $\whMwhg$~is a Riemannian manifold with~a \emph{wedge singularity 
near}~$\Sa$.\po  

\par 
The following theorem is the basis for our approach. It implies that it 
suffices to study transmission problems for membranes without boundary on 
urR manifolds. 
\begin{theorem}\label{thm-S.M} 
$\whMwhg$ is an oriented urR manifold and 
\hb{\wh{S}:=S\ssm\Sa} is a membrane in~$\wh{M}$ without boundary. 
\end{theorem} 
\begin{proof} 
(1) 
We set 
\hb{\dot D(\ve):=D(\ve)\ssm\{0,0\}} and define 
\hb{\da\in C^\iy[0,\ve)} by 
$$ 
\rho(x,y)=\da\bigl(r(x,y)\bigr) 
\qa (x,y)\in D(\ve). 
$$ 
Then we fix 
\hb{\wh{\ve}\in(2\ve/3,\,\ve)}, define a diffeomorphism 
$$ 
s\sco(0,\wh{\ve}\,]\ra\BR_+ 
\qb r\mt\int_r^{\wh{\ve}}\frac{dt}{\da(t)}, 
$$ 
and set 
\hb{t:=s^{-1}}. It follows, see \cite[Lemma~5.1]{Ama20a}, that 
$$ 
t^*\Bigl(\frac{dr^2}{\da^2(r)}\Bigr)=ds^2. 
$$ 
We also consider the polar coordinate diffeomorphism 
$$ 
R\sco(0,\wh{\ve}\,)\times[-\pi/2,\,\pi/2]\ra\dot D(\wh{\ve}) 
\qb (r,\al)\mt r(\cos\al,\sin\al)\ . 
$$ 
Then 
$$ 
R^*(dx^2+dy^2)=dr^2+r^2d\al^2 
=\da^2\Bigl(\frac{dr^2}{\da^2}+\frac{r^2}{\da^2}\,d\al^2\Bigr), 
$$ 
that is, 
\begin{equation}\label{S.Rg} 
R^*\Bigl(\frac{dx^2+dy^2}{\rho^2}\Bigr) 
=\frac{dr^2}{\da^2}+\frac{r^2}{\da^2}\,d\al^2. 
\end{equation}  
Hence 
$$ 
\lda:=R\circ(t\times\id) 
\sco(0,\iy)\times[-\pi/2,\,\pi/2]\ra\dot D(\wh{\ve}) 
$$ 
is a diffeomorphism satisfying 
$$ 
\lda^*\Bigl(\frac{dx^2+dy^2}{\rho^2}\Bigr) 
=(t\times\id)^*R^*\Bigl(\frac{dx^2+dy^2}{\rho^2}\Bigr) 
=ds^2+\ba^2(s)d\al^2=:\ga^2, 
$$ 
where 
\hb{\ba:=t^*(r/\da)}. By~\eqref{R.rh}, 
\hb{r/\da=r/(1-\om+r\om)} for 
\hb{0<r\leq\wh{\ve}}. 
Hence $\ba$~is smooth and 
\hb{\ba\sim1}. Thus $\ga$~is a metric on 
\hb{N:=(0,\iy)\times[-\pi/2,\,\pi/2]} which is uniformly equivalent to 
\hb{g_2=ds^2+d\al^2}. By Examples~\ref{exa-U.ex} \hbox{(a)--(c)}, 
$$ 
\bigl(\BR\times[-\pi/2,\,\pi/2],\,ds^2+d\al^2\bigr) 
$$ 
is a urR manifold. From this we deduce, see Remark~\ref{rem-U.R}(a), that 
$(N,\ga)$~is a urR manifold on~$(\ol{s},\iy)$ for each 
\hb{\ol{s}>0}.  

\par 
It follows that 
$$ 
w:=(\lda^{-1}\times\id)\circ\chi\sco\bigl(U(\wh{\ve}),\wh{g}\bigr) 
\ra(N\times\Sa,\,\ga+h) 
$$ 
is an isometric isomorphism. Hence we derive from Example~\ref{exa-U.ex}(d) 
and Remark~\ref{rem-U.R}(a) that $U(\wh{\ve})$~is a urR manifold 
on~$U(\ol{\ve})$, where 
\hb{\ol{\ve}:=5\ve/6}. Since $\Mg$~is a urR manifold, it is a urR manifold 
on 
\hb{M\big\backslash\wt{U}(\ve/3)}. Thus it is a consequence of \eqref{S.gg} 
that $\whMwhg$~is a urR manifold. The first assertion is now clear. 

\par 
(2) 
Fix 
\hb{\ol{\ve}\in(\ve/3,\,\wh{\ve})}. It can be 
assumed that \eqref{H.f}~applies with this choice of~$\ol{\ve}$. Set 
\hb{\wt{f}_\sa:=t^*f_\sa}. Note that \eqref{H.f}(ii) implies 
\begin{equation}\label{S.fpb} 
\wt{f}_\sa\sco\bigl[s(\ol{\ve}),\iy\bigr)\ra[-\ol{\rho},\ol{\rho}] 
\qa \sa\in\Sa. 
\end{equation}  
Also note that 
\hb{t(s)=c\coc e^{-s}} for 
\hb{s\geq s(\ve/3)} and some 
\hb{c>0}. Hence 
$$ 
\pl\wt{f}_\sa(s)=-c\coc e^{-s}\pl f_\sa(c\coc e^{-s}) 
\qa s\geq s(\ve/3). 
$$ 
Thus it follows from \eqref{H.f}(iii) that 
\begin{equation}\label{S.df} 
\lim_{s\ra\iy}\pl\wt{f}_\sa(s)=0 
\qquad\text{$\sa$-unif.}, 
\npbd 
\end{equation}  
that is, uniformly with respect to 
\hb{\sa\in\Sa}.\po  

\par 
We write~$\wt{G}_\sa$ for the graph of~$\wt{f}_\sa$ in 
\hb{\bigl[s(\ol{\ve}),\iy\bigr)\times[-\pi/2,\,\pi/2]}. We can assume that 
\begin{equation}\label{S.nu} 
\wt{\nu}_\sa(s) 
:=\frac{(\pl\wt{f}_\sa(s),-1)}{(1+|\pl\wt{f}_\sa(s)|^2)^{1/2}} 
\qa s\geq s(\ol{\ve}), 
\end{equation}  
is the positive normal for~$\wt{G}_\sa$ at~$\bigl(s,\wt{f}_\sa(s)\bigr)$ 
(otherwise replace~$\wt{\nu}_\sa(s)$ by~$-\wt{\nu}_\sa(s)$). 
It follows from~\eqref{S.df} that 
\begin{equation}\label{S.n} 
\lim_{s\ra\iy}\wt{\nu}_\sa(s)=(0,-1) 
\qquad\text{$\sa$-unif.} 
\end{equation}  
From this and \eqref{S.fpb} we deduce that $\wt{G}_\sa$~has a uniform tubular 
neighborhood in 
\hb{\bigl([s(\ol{\ve}),\iy)\times[-\pi/2,\,\pi/2],\,ds^2+d\al^2\bigr)} 
whose width is independent of 
\hb{\sa\in\Sa}. It follows from step~(1) that its pull-back by~$w$ is a 
uniform tubular neighborhood of~$\wh{S}$ 
in~$U(\ol{\ve})$. Now the second part of the assertion 
is a consequence of Theorem~\ref{thm-H.S}(iv), since, given any 
\hb{\da>0},
\,$\wh{g}$~and~$g$ are equivalent on 
\hb{M\big\backslash\wt{U}(\da)}. 
\end{proof}
\section{Function Spaces}\label{sec-F} 
Let $\Mg$ be a Riemannian manifold. We consider the tensor bundles 
$$ 
T_0^1M:=TM 
\qb T_1^0M:=T^*M 
\qb T_0^0=M\times\BR, 
$$ 
the tangent, cotangent, and trivial bundle, respectively, and 
$$ 
T_\tau^\sa M:=(TM)^{\otimes\sa}\otimes(T^*M)^{\otimes\tau} 
\qa \sa,\tau\geq1,  
$$ 
endow~$T_\tau^\sa M$  with the tensor bundle metric 
\hb{g_\sa^\tau:=g^{\otimes\sa}\otimes g^{*\,\otimes\tau}}, 
\,\hb{\sa,\tau\in\BN}, and set\footnote{If $V$~is a vector bundle over~$M$, 
then $C^k(V)$~denotes the vector space of $C^k$~sections of~$V$.} 
\begin{equation}\label{F.gst} 
|a|_{g_\sa^\tau}=\sqrt{(a\sn a)_{g_\sa^\tau}} 
:=\sqrt{g_\sa^\tau(a,a)} 
\qa a\in C(T_\tau^\sa M). 
\end{equation}  
By 
\hb{\na=\nag} we denote the Levi--Civita connection and interpret it as 
covariant derivative. Then, given a smooth function~$u$ on~$M$, 
\ \hb{\na^ku\in C^\iy(T_k^0M)} is defined by 
\hb{\na^0u:=u}, 
\ \hb{\na^1u=\na u:=du}, and 
\hb{\na^{k+1}u:=\na(\na^ku)} for 
\hb{k\in\BN}. 

\par 
Let 
\hb{\ka=(x^1,\ldots,x^m)} be a local coordinate system and set 
\hb{\pl_i:=\pl/\pl x^i}. Then 
$$ 
\na^1u=\pl_iu\,dx^i 
\qb \na^2u=\na_{\cona ij}u\,dx^i\otimes dx^j 
 =(\pl_i\pl_ju-\Ga_{ij}^k\pl_ku)dx^i\otimes dx^j,
$$ 
where 
$$ 
\Ga_{ij}^k=\frac12\,g^{k\ell}(\pl_ig_{j\ell}+\pl_jg_{i\ell}-\pl_\ell g_{ij}) 
\qa 1\leq i,j,k\leq m, 
$$ 
are the Christoffel symbols. It follows that 
\begin{equation}\label{F.1} 
|\na u|_{g_0^1}^2=|\na u|_{g^*}^2=g^{ij}\pl_iu\pl_ju 
\end{equation}  
and 
\begin{equation}\label{F.2} 
|\na^2u|_{g_0^2}^2 
=g^{i_1j_1}g^{i_2j_2}\na_{\cona i_1i_2}u\na_{\cona j_1j_2}u. 
\npbd 
\end{equation}  
As usual, 
\hb{d\vol_g=\sqrt{g}\,dx} is the Riemann--Lebesgue volume element 
on~$U_{\coU\ka}$.\po  

\bigskip 
Let 
\hb{\sa,\tau\in\BN}, put 
\hb{V:=T_\tau^\sa M}, and write 
\hb{\vsdot_V:=\vsdot_{g_\sa^\tau}}. Then $\cD(V)$~is the linear subspace 
of~$C^\iy(V)$ of compactly supported sections. 

\par 
For 
\hb{1\leq q<\iy} we set 
$$ 
\|u\|_{L_q(V)}=\|u\|_{L_q\Vg} 
:=\Bigl(\int_M|u|_V^q\,d\vol_g\Bigr)^{1/q}. 
$$ 
Then 
$$ 
L_q(V)=L_q\Vg 
:=
\bigl(\bigl\{\,u\in L_{1,\loc}(M)\ ;\ \Vsdot_{L_q\Mg}<\iy 
\,\bigr\},\ \Vsdot_{L_q\Mg}\bigr) 
$$ 
is the usual Lebesgue space of \hbox{$L_q$~sections} of~$V$. Hence 
\hb{L_q\Mg=L_q\Vg} for 
\hb{V=T_0^0M=M\times\BR}. If 
\hb{k\in\BN}, then 
$$ 
\|u\|_{W_{\coW q}^k(V)}=\|u\|_{W_{\coW q}^k\Vg} 
:=\sum_{j=0}^k\big\|\,|\na^jv|_{g_\sa^{\tau+j}}\big\|_{L_q\Mg} 
$$ 
and 
$$ 
\|u\|_{BC^k(V)}=\|u\|_{BC^k\Vg} 
:=\sum_{j=0}^k\big\|\,|\na^jv|_{g_\sa^{\tau+j}}\big\|_\iy. 
$$ 
The Sobolev space 
\hb{W_{\coW q}^k(V)=W_{\coW q}^k\Vg} is the 
completion of $\cD(V)$ in~$L_q(V)$ with respect to the norm~%
\hb{\Vsdot_{W_{\coW q}^k(V)}}. If 
\hb{k<s<k+1}, the Slobodeckii space $W_{\coW q}^s(V)$ is obtained by real 
interpolation: 
\begin{equation}\label{F.Ws} 
W_{\coW q}^s(V)=W_{\coW q}^s\Vg 
:=\bigl(W_{\coW q}^k(V),W_{\coW q}^{k+1}(V)\bigr)_{s-k,q}. 
\end{equation}  

\par 
We also need the time-dependent function spaces 
\begin{equation}\label{F.WJ} 
W_{\coW q}^{s/\mf2}(M\times J) 
:=L_q\bigl(J,W_{\coW q}^s(M)\bigr)\cap W_{\coW q}^{s/2}\bigl(J,L_q(M)\bigr) 
\qa 0\leq s\leq2, 
\npbd 
\end{equation}  
Thus 
\hb{W_{\coW q}^{0/\mf2}(M\times J)\doteq L_q\bigl(J,L_q(M)\bigr)}, 
where 
\hb{{}\doteq{}}~means `equivalent norms'.\po  

\par 
By 
\hb{BC^k(V)=BC^k\Vg} we denote the Banach space of all 
\hb{u\in C^k(V)} for which 
\hb{\|u\|_{BC^k(V)}}~is finite, and 
\hb{BC:=BC^0}. Then 
\begin{equation}\label{F.BJ} 
BC^{1/\mf2}(M\times J) 
:=C\bigl(J,BC^1(M)\bigr)\cap C^{1/2}\bigl(J,BC(M)\bigr) 
\npbd 
\end{equation} 
with the usual H\"older space~$C^{1/2}$.\po   

\par
The following lemma shows that in the Euclidean setting these definitions 
return the classical spaces. 
\begin{lemma}\label{lem-F.X} 
Suppose that   
\hb{\BX\in\{\BR^m,\BH^m\}}, 
\hb{\Mg:=(\BX,g_m)}, and  
\hb{V:=\BX\times F} with 
\hb{F=\BR^{m^\sa\times m^\tau}\simeq T_\tau^\sa\BX}. Then 
$$ 
W_{\coW q}^s(V)=W_{\coW q}^s\BXF 
\qa s\in\BR_+ 
\qb 1\leq q<\iy, 
$$ 
the standard Sobolev--Slobodeckii spaces, and 
$$ 
BC^k(V)=BC^k\BXF 
\qa k\in\BN. 
$$ 
\end{lemma}\po 
\begin{proof} 
The second assertion is obvious. 

\par 
If 
\hb{k\in\BN}, then the above definition of~$W_{\coW q}^k(V)$ coincides with 
the one in~\cite[(VII.1.2.2)]{Ama19a}. 
Now the first assertion follows 
from~\eqref{F.Ws}, Theorems VII.2.7.4 and VII.2.8.3 
in~\cite{Ama19a}, 
and the fact that the Besov space 
\hb{B_q^s=B_{qq}^s} coincides with~$W_{\coW q}^s$ for 
\hb{s\notin\BN}. 
\end{proof} 
Now we suppose that 
\begin{equation}\label{F.S} 
\bal 
{}
&\Mg\text{ is an oriented urR manifold and}\cr 
\noalign{\vskip-1\jot}
&S\text{ is a membrane without boundary}.
\eal 
\end{equation}  
By Theorem~\ref{thm-H.S}(iii) we can chose a uniform tubular neighborhood 
\begin{equation}\label{F.pS} 
\vp\sco S(\ve)\ra(-\ve,\ve)\times S 
\end{equation}  
in~$\ci M$. We set 
$$ 
M_+:=\vp^{-1}\bigl([0,\ve)\times S\bigr) 
\qb M_-:=\vp^{-1}\bigl((-\ve,0]\times S\bigr)  
\qb M_0:=M\big\backslash\ol{S(\ve/2)}. 
$$ 
By 
\hb{V_\pm:=V_{M_\pm}} and 
\hb{V_0:=V_{M_0}} we denote the restrictions of~$V$ to $M_\pm$ and~$M_0$, 
respectively. Then 
\hb{\bar W_{\coW q}^s(M\ssm S,\,V)}, 
\,\hb{s\in\BR_+}, resp.\ 
\hb{\bar{BC}^k(M\ssm S,\,V)}, 
\,\hb{k\in\BN}, is the closed linear subspace of 
$$ 
W_{\coW q}^s(V_0)\oplus W_{\coW q}^s(V_+)\oplus W_{\coW q}^s(V_-), 
\quad\text{resp.}\quad 
BC^k(V_0)\oplus BC^k(V_+)\oplus BC^k(V_-), 
$$ 
consisting of all 
\hb{u=u_0\oplus u_+\oplus u_-} satisfying 
\hb{(u_0-u_\pm)\sn M_0\cap M_\pm=0}. 
Definitions analogous to \eqref{F.WJ} and \eqref{F.BJ} give 
the Banach spaces 
\hb{\bar W_{\coW q}^{s/\mf2}\bigl((M\ssm S)\times J\bigr)} and 
\hb{\bar{BC}^{1/\mf2}\bigl((M\ssm S)\times J\bigr)}, respectively. 
Note that 
\hb{\bar W_{\coW p}^0(M\ssm S,\,V)=L_p\MV}, since 
\hb{\vol_g(S)=0}.  
\setcounter{remark}{1}
\begin{remark}\label{rem-F.H} 
Let 
\hb{\Mg:=(\BR^m,g_m)} and 
\hb{S:=\pl\BH^m}. We can set 
\hb{\ve=\iy} in \eqref{F.pS} to get 
\hb{M_+=\BH^m}, 
\ \hb{M_-=-\BH^m}, and 
\hb{M_0=\es}. Then 
$$ 
\bar W_{\coW q}^s(M\ssm S,\,V) 
=W_{\coW q}^s(\BH^m,F)\oplus W_{\coW q}^s(-\BH^m,F) 
$$ 
and 
$$ 
\bar{BC}^k(M\ssm S,\,V) 
=BC^k(\BH^m,F)\oplus BC^k(-\BH^m,F), 
\npbd  
$$ 
as follows from Lemma~\ref{lem-F.X}.\qed 
\end{remark} 
Assume 
\hb{a\in\bar{BC}(M\ssm S,\,V)}. Then the one-sided limits 
$$ 
\lim_{t\ra0\pm}a\bigl(\ga_p^S(t)\bigr)=:a_\pm(p) 
\qa p\in S, 
$$ 
exist and 
\hb{a_\pm\in BC(S)}. Hence the \emph{jump across}~$S$, 
$$ 
\eea{a}:=\bigl(p\mt\eea{a}(p) 
:=a_+(p)-a_-(p)\bigr)\in BC(S), 
$$ 
is well-defined. Note that $a_\pm$~is the trace of~$a$ on~$S$ `from the 
positive/negative side of~$S$'.   

\par 
Let 
\hb{u\in\bar{BC}^1(M\ssm S)}. Then 
\hb{u\circ\ga_p^S\in\bar{BC}^1\bigl((-\ve,\ve)\ssm\{0\}\bigr)}. We set 
$$ 
\frac{\pl u}{\pl\nu_S}(q) 
:=\pl_1(u\circ\vp^{-1})(\tau,p) 
\qa q\in (M_+\cup M_-)\ssm S, 
$$ 
for 
\hb{\vp(q)=(\tau,p)\in(-\ve,\ve)\times S} with 
\hb{\tau\neq0}. Thus 
\hb{\pl u/\pl\nu_S} is the \emph{normal derivative} of~$u$ in 
\hb{(M_+\cup M_-)\ssm S}, that is, the derivative along the normal 
geodesic~$\ga_p^S$. Hence 
$$ 
\frac{\pl u}{\pl\nu_S}(q) 
=\bigl\dl du(q),\dot\ga_p^S(\tau)\bigr\dr 
=\bigl(\dot\ga_p^S(\tau)\bsn\grad u(q)\bigr)_{g(q)}. 
$$ 
Consequently, the \emph{jump of the normal derivative}, 
$$ 
\Beea{\displaystyle\frac{\pl u}{\pl\nu_S}} 
=\beea{(\nu_S\sn\grad u)} 
\in BC(S) 
\npbd 
$$ 
is also well-defined.\po  
\section{The Parabolic Problem on Manifolds}\label{sec-L} 
We presuppose \eqref{F.S} and assume 
\begin{equation}\label{L.a} 
\bal 
{\rm(i)}\quad 
&a\in\bar{BC}^{1/\mf2}\bigl((M\ssm S)\times J\bigr).\cr 
{\rm(ii)}\quad 
&a\geq\ul{\al}>0, 
\eal 
\end{equation}  
where 
\hb{\ul{\al}\leq1}. Then 
$$ 
\cA u:=-\tdiv(a\grad u). 
$$ 

\par 
Fix 
\hb{\da\in C\bigl(\Ga,\{0,1\}\bigr)}. Then 
\hb{\Ga_j:=\da^{-1}(j)}, 
\,\hb{j=0,1}, is open and closed in~$\Ga$ and 
\hb{\Ga_0\cup\Ga_1=\Ga}. Either $\Ga$, $\Ga_0$, or~$\Ga_1$ may be empty. 
In such a case all references to these empty sets have to be disregarded. 
Recall that $\ga$~denotes the trace operator on~$\Ga$ (for any manifold).  

\par 
We introduce an operator 
\hb{\cB=(\cB^0,\cB^1)} on~$\Ga$ by 
\hb{\cB^0u=\ga_{\Ga_0}u}, the Dirichlet boundary operator, on~$\Ga_0$, and 
a~Neumann boundary operator 
$$ 
\cB^1u:=\bigl(\nu\bsn\ga_{\Ga_1}(a\grad u)\bigr)\text{ on }\Ga_1. 
$$ 

\par 
On~$S$ we consider the transmission operator 
\hb{\cC=(\cC^0,\cC^1)}, where 
$$ 
\cC^0u:=\eea{u} 
\qb \cC^1u:=\beea{(\nu_S|a\grad u)}. 
\npbd 
$$ 
Note that 
\hb{\beea{(\nu_S\sn a\grad u)}} equals 
\hb{\eea{a\pl_{\nu_S}u}} and not 
\hb{\beea{\pl_{\nu_S}(au)}}.\po  

\par 
Of concern in this paper is the inhomogeneous linear transmission 
problem 
\begin{equation}\label{L.T} 
\bal 
\pl_tu+\cA u    &=f     &&\text{ on }(M\ssm S)\times J,\cr 
\cB u           &=\vp   &&\text{ on }\Ga\times J,\cr 
\cC u           &=\psi  &&\text{ on }S\times J,\cr 
\ga_0u          &=u_0   &&\text{ on }(M\ssm S)\times\{0\}.  
\eal 
\end{equation}\po  

\par 
We assume that 
\begin{equation}\label{L.ass} 
1<p<\iy,\ p\notin\{3/2,\,3\}, 
\ \text{\eqref{F.S} and \eqref{L.a} are satisfied}. 
\end{equation}  

\par 
For abbreviation we set, for 
\hb{1<q<\iy},  
$$ 
\bar W_{\coW q}^{k/\mf2} 
:=\bar W_{\coW q}^{k/\mf2}\bigl((M\ssm S)\times J\bigr) 
\qa k=0,2, 
$$ 
and introduce the trace spaces 
$$ 
\bal 
\pl W_{\coW q}              
&:=W_{\coW q}^{(2-1/q)/\mf2}(\Ga_0\times J) 
 \oplus W_{\coW q}^{(1-1/q)/\mf2}(\Ga_1\times J),\cr 
\pl_SW_{\coW q}               
&:=W_{\coW q}^{(2-1/q)/\mf2} 
 (S\times J)\oplus W_{\coW q}^{(1-1/q)/\mf2}(S\times J), 
\eal 
$$ 
and 
$$ 
\ga_0\bar W_{\coW q} 
:=\bar W_{\coW q}^{2-2/q}(M\ssm S). 
$$ 
As a rule, we often drop the index~$q$ if 
\hb{q=p}. Thus 
\hb{\bar W^{2/\mf2}=\bar W_{\coW p}^{2/\mf2}}, 
\ \hb{\pl W=\pl W_{\coW p}}, etc. Finally, 
$$ 
\pl_{\cB,\cC}W 
=[\pl W\oplus\pl_SW\oplus\ga_0\bar W]_{\cB,\cC} 
$$ 
is the closed linear subspace of 
\hb{\pl W\oplus\pl_SW\oplus\ga_0\bar W} consisting of all 
$(\vp,\psi,u_0)$ satisfying the compatibility conditions 
$$ 
\bal 
\cB^0u_0 
&=\vp^0(0),     &\quad \cC^0u_0     &=\psi^0(0), 
&&\quad\text{if}\quad 
&3/2   &<p<3,\cr 
\cB(0)u_0 
&=\vp(0),       &\quad \cC(0)u_0    &=\psi(0), 
&&\quad\text{if}\quad 
&3     &<p<\iy,  
\eal 
\npbd 
$$ 
where 
\hb{\vp=(\vp^0,\vp^1)} and 
\hb{\psi=(\psi^0,\psi^1)}. It follows from the anisotropic trace theorem 
(\cite[Example VIII.1.8.6]{Ama19a}) that $\pl_{\cB,\cC}W$~is well-defined. 

\par 
Given Banach spaces $E$ and~$F$, we  write~$\Lis\EF$ for the set
of all isomorphisms in~$\cL\EF$, the Banach space of 
continuous linear maps from~$E$ into~$F$. 

\par 
Now we can formulate the following maximal regularity theorem for 
problem~\eqref{L.T}. Its proof, which needs considerable preparation, 
is found in Section~\ref{sec-P}. 
\begin{theorem}\label{thm-L.MR} 
Let \eqref{L.ass} be satisfied. Then 
$$ 
\bigl(\pl_t+\cA,\,(\cB,\cC,\ga_0)\bigr) 
\in\Lis\bigl(\bar W_{\coW p}^{(2,1)}, 
L_p(J,L_p(M))\times\pl_{\cB,\cC}W_{\coW p}\bigr). 
$$
\end{theorem} 
\section{The Uniform Lopatinskii--Shapiro Condition}\label{sec-E} 
In the proof of Theorem~\ref{thm-L.MR} we need to consider systems of 
elliptic boundary value problems. For this we have to be precise on the 
concept of uniform ellipticity. 

\par 
Let $\Mg$ be any Riemannian manifold. We consider a general second order 
linear differential operator on~$M$, 
\begin{equation}\label{E.A} 
\cA u:=-a\btdot\na^2u+a_1\btdot\na u+a_0u 
\end{equation}  
with 
\hb{u=(u^1,\ldots,u^n)} and 
$$ 
a_i\in C(T_0^iM)^{n\times n} 
\qa i=0,1,2 
\qb a_2:=a. 
$$ 
Here 
\hb{\na^iu=(\na^iu^1,\ldots,\na^iu^n)} so that, for example, 
$$ 
a\btdot\na^2u=(a_s^1\btdot\na^2u^s,\ldots,a_s^n\btdot\na^2u^s), 
$$ 
where $s$~is summed from~$1$ to~$n$ and 
\hb{{}\btdot{}}~denotes complete contraction, that is, summation over all 
twice occurring indices in any local coordinate representation. 

\par 
The (principal) symbol~$\gss\cA$ of~$\cA$ is the 
\hb{(n\times n)}-matrix-valued function defined by 
$$ 
\gss\cA(p,\xi):=a(p)\btdot(\xi\otimes\xi) 
\qa p\in M 
\qb \xi\in T_p^*M. 
$$ 
Then $\cA$~is \emph{uniformly normally elliptic} if there exists an 
`ellipticity constant' 
\hb{\ul{\al}\in(0,1)} such that 
$$ 
\sa\bigl(\gss\cA(p,\xi)\bigr)\is[\Re z\geq\ul{\al}] 
=\{\,z\in\BC\ ;\ \Re z\geq\ul{\al}\,\} 
\npbd 
$$ 
for all 
\hb{p\in M} and 
\hb{\xi\in T_p^*M} with 
\hb{|\xi|_{g^*(p)}^2=1}, where $\sa(\cdot)$~denotes the spectrum.\po  

\par 
Suppose 
\hb{\Ga\neq\es} and 
\hb{\cB=(\cB^1,\ldots,\cB^n)} is a linear boundary operator of order 
at most~$1$. More precisely, we assume that there is 
\hb{k\in\{0,\ldots,n\}} such that 
$$ 
\cB^ru=
\left\{ 
\bal 
{}
&b_0^r\btdot\ga u, 
&\quad 1    &\leq r\leq k,\cr 
&b_1^r\btdot\ga\na u+b_0^r\btdot\ga u, 
&\quad k+1  &\leq r\leq n,\cr 
\eal 
\right. 
$$ 
with 
$$ 
b_i^r\in C(T_0^i\Ga)^n 
\qa 1\leq r\leq n 
\qb i=0,1. 
$$ 
Then the (principal) symbol of~$\cB$ is the 
\hb{(n\times n)}-matrix-valued function~$\gss\cB$ given by 
$$ 
\gss\cB^r(q,\xi):=
\left\{ 
\bal 
{}
&b_0^r(q),\cr 
&b_1^r(q)\btdot\xi,  
\eal 
\right. 
\qquad q\in\Ga 
\qb \xi\in T_q^*M. 
$$ 
Observe that 
\hb{X\btdot\om=\dl\om,X\dr} if $X$~is a vector, $\om$~a covector field, and 
\hb{\pw} the duality pairing. 

\par 
We denote by 
\hb{\nu^\flat\in T_\Ga^*M} the inner conormal on~$\Ga$ defined in local 
coordinates by 
\hb{\nu^\flat=g_{ij}\nu^jdx^i}. Given 
\hb{q\in\Ga}, we write~$\BB(q)$ for the set of all 
\begin{equation}\label{E.Bq} 
\bal 
{}
&(\xi,\lda)\in T_q^*M\times[\Re z\geq0]\text{ satisfying}\cr
&\xi\perp\nu^\flat(q)\text{ and }|\xi|^2+|\lda|=1. 
\eal 
\end{equation}  
Then, if 
\hb{(\xi,\lda)\in\BB(q)}, we introduce linear differential  operators 
on~$\BR$ by 
\begin{equation}\label{E.AB} 
\bal 
A(\pl;q,\xi,\lda) 
&:=\lda+\gss\cA\bigl(q,\,\xi+\imi\nu^\flat(q)\pl\bigr),\cr 
B(\pl;q,\xi,\lda) 
&:=\gss\cB\bigl(q,\,\xi+\imi\nu^\flat(q)\pl\bigr), 
\eal 
\npbd 
\end{equation}  
where 
\hb{\imi=\sqrt{-1}}.\po  

\par 
As usual, $C_0(\BR_+,\BC^n)$ is the closed linear subspace 
of $BC(\BR_+,\BC^n)$ consisting of the functions that vanish at infinity. 

\par 
Suppose $\cA$~is uniformly normally elliptic. Then it follows that the 
homogeneous problem 
\begin{equation}\label{E.Aa} 
A(\pl;q,\xi,\lda)v=0\text{ on }\BR 
\end{equation}  
has for each 
\hb{q\in\Ga} and 
\hb{(\xi,\lda)\in\BB(q)} precisely~$n$ linearly independent solutions 
whose restrictions to~$\BR_+$ belong to~$C_0(\BR_+,\BC^n)$. 
We denote their span by $C_0(q,\xi,\lda)$. It is an \hbox{$n$-dimensional} 
linear subspace of $C_0(\BR_+,\BC^n)$. 

\par 
Now we consider the initial value problem on the half-line: 
\begin{equation}\label{E.LS} 
\bal 
A(\pl;q,\xi,\lda)v      &=0\text{ on }\BR_+,\cr 
B(\pl;q,\xi,\lda)v(0)   &=\eta\in\BC^n.   
\eal 
\end{equation}  
Then $\AB$~satisfies the \emph{uniform parameter-dependent 
Lopatinskii-Shapiro} (LS) \emph{conditions} 
if problem~\eqref{E.LS} has for each 
\hb{\eta\in\BC^n} a~unique solution 
\begin{equation}\label{E.v} 
v=R(q,\xi,\lda)\eta\in C_0(q,\xi,\lda) 
\npbd 
\end{equation}  
and 
\begin{equation}\label{E.R} 
\|R(q,\xi,\lda)\|_{\cL(\BC^n,C_0(\BR_+,\BC^n))}\leq c, 
\npbd 
\end{equation}  
unif.\ w.r.t.\  
\hb{q\in\Ga} and 
\hb{(\xi,\lda)\in\BB(q)}.\po   

\par 
The basic feature, which distinguishes the above definition from the usual 
form of the LS~condition, is the requirement of the uniform 
bound~\eqref{E.R}. Without this requirement the LS~condition is much 
simpler to 
formulate (e.g.,~\cite{Ama90a}, \cite{Ama93a}, \cite{DHP03a}, 
\cite{DHP07a}, \cite{PrS16a}, for example) and to verify. 

\par 
It is known that the LS~condition is equivalent to the parameter-dependent 
version of the complementing condition of 
S.~Agmon, A.~Douglis, and L.~Nirenberg~\cite{ADN64a} (see, for example, 
\cite[VII\S9]{LSU68a} or \cite[Section~10.1]{WRL95a}). Using this version, 
it is possible to define a uniform complementing condition which is 
equivalent to~\eqref{E.R} (see \cite{Ama84b} and~\cite{Ama85a}). 
However, that condition is even more difficult to verify in concrete 
situations. We refer to~\cite{AmaVolIII21a} for a detailed exposition of 
all these facts. It should be noted that the uniformity 
condition~\eqref{E.R} is fundamental for the following, since we will have 
to work with infinitely many linear model problems. 
\section{Model Cases}\label{sec-M} 
For the proof of Theorem~\ref{thm-L.MR} we have to understand the model 
cases to which problem~\eqref{L.T} reduces in local coordinates. 

\par 
Until further notice, it is assumed that 
$$ 
\bal 
\bt\quad 
&\text{assumption \eqref{L.ass} applies}.\cr 
\bt\quad 
&\gK\text{ is an $S$-adapted ur atlas for }M.  
\eal 
$$ 
By Remark~\ref{rem-U.R}(b) we can assume that 
\hb{\diam(U_{\coU\ka})<\ve/2} for 
\hb{\ka\in\gK_S} where $\ve$~is the width of the tubular neighborhood 
of~$S$. 

\par 
We can choose a family 
\hb{\{\,\pi_\ka,\chi\ ;\ \ka\in\gK\,\}} with the following properties:   
\begin{equation}\label{M.LS} 
\bal 
{\rm(i)}\quad 
&\pi_\ka\in\cD\bigl(U_{\coU\ka},[0,1]\bigr)\text{ for $\ka\in\gK$ and } 
 {\textstyle\sum_\ka}\pi_\ka^2(p)=1\text{ for }p\in M.\cr 
{\rm(ii)}\quad 
&\|\ka_*\pi_\ka\|_{k,\iy}\leq c(k),\ \ka\in\gK,\ k\in\BN.\cr 
{\rm(iii)}\quad 
&\chi\in\cD\bigl((-1,1)^m,[0,1]\bigr) 
 \text{ with }\chi\sn\supp(\ka_*\pi_\ka)=1\text{ for }\ka\in\gK. 
\eal 
\end{equation}  
(See Lemma~3.2 in~\cite{Ama12b} or~\cite{AmaVolIII21a}). We fix an 
\hb{\wt{\om}\in\cD\bigl((-1,1)^m,[0,1]\bigr)} which is equal to~$1$ on  
$\supp(\chi)$. Then 
$$ 
g_\ka:=\wt{\om}\ka_*g+(1-\wt{\om})g_m 
$$ 
is a Riemannian metric on~$\BR^m$ such that 
\begin{equation}\label{M.ka} 
g_\ka\sim g_m 
\qa \ka\in\gK, 
\end{equation}  
and 
\begin{equation}\label{M.g} 
\|g_\ka\|_{k,\iy}\leq c(k) 
\qa \ka\in\gK 
\qb k\in\BN. 
\end{equation}  
This follows from \eqref{U.g}. Furthermore, 
\begin{equation}\label{M.a} 
a_\ka:=\wt{\om}\ka_*a+1-\wt{\om}. 
\end{equation}  
Note that 
\begin{equation}\label{M.aa} 
a_\ka\geq\ul{\al} 
\qa \ka\in\gK. 
\end{equation}  

\par 
We write 
\hb{\grad_\ka:=\grad_{g_\ka}} and 
\hb{\tdiv_\ka:=\tdiv_{g_\ka}} for 
\hb{\ka\in\gK}. Then 
$$ 
\cA_\ka u:=-\tdiv_\ka(a_\ka\grad_\ka u) 
\qa u\in Q_\ka^m. 
$$ 
Let 
\hb{\da_\ka:=\ka_*\da}. Then 
$$ 
\cB_\ka u 
:=\da_\ka\bigl(\nu_\ka\bsn\ga(a_\ka\grad_\ka u)\bigr)_\ka+(1-\da_\ka)\ga u 
\qa \ka\in\gK_\Ga, 
$$ 
where $\nu_\ka$~is the inner normal on~$\pl\BH^m$ with respect to 
$(\BH^m,g_\ka)$, and 
\hb{\prsn_\ka=g_\ka}. If 
\hb{\ka\in\gK_S}, then 
$$ 
\cC_\ka u:=\Bigl(\eea{u},\beea{(\nu_\ka\sn a_\ka\grad_\ka u)_\ka}\Bigr) 
\text{ on }\pl\BH^m.  
$$ 
Using these notations, we consider the three model problems: 
\begin{equation}\label{M.P1} 
\pl_tu+\cA_\ka u=f_\ka\text{ on }\BR^m\times J,
\end{equation}  
and 
\begin{equation}\label{M.P2} 
\bal 
\pl_tu+\cA_\ka u&=f_\ka     &&\text{ on }\BH^m\times J,\cr 
\cB_\ka u       &=\vp_\ka   &&\text{ on }\pl\BH^m\times J,
\eal 
\end{equation}  
and 
\begin{equation}\label{M.P3} 
\bal 
\pl_tu+\cA_\ka u&=f_\ka     &&\text{ on }(\BR^m\ssm\pl\BH^m)\times J,\cr 
\cC_\ka u       &=\psi_\ka  &&\text{ on }\pl\BH^m\times J. 
\eal 
\end{equation}  

\par 
In the following two sections we prove that each one of them, 
complemented by appropriate initial and compatibility conditions, enjoys a 
maximal regularity result, unif.\ w.r.t.~$\ka$. 
\section{Continuity}\label{sec-C} 
First we note that
\begin{equation}\label{C.g} 
\sqrt{g_\ka}\sim1 
\qa \ka\in\gK, 
\end{equation}  
and, given 
\hb{k\in\BN}, 
\begin{equation}\label{C.Nd} 
\sum_{i=0}^k|\naka^iu|\sim\sum_{|\al|\leq k}|\pa u| 
\qa \ka\in\gK 
\qb u\in C^k(Q_\ka^m), 
\npbd 
\end{equation}  
with 
\hb{\naka u:=\ka_*\na\ka^*u} 
(cf.~\cite[Lemma~3.1]{Ama12b} or~\cite{AmaVolIII21a}). 

\par 
We set 
\begin{equation}\label{C.X} 
\BX_\ka:=  
\left\{ 
\bal 
{}
&\BR^m,
 &&\quad \text{if }\ka\in\gK_0:=\gK\ssm(\gK_\Ga\cup\gK_S),\cr 
&\BH^m,
 &&\quad \text{if }\ka\in\gK_\Ga,\cr 
&\BR^m\ssm\pl\BH^m,
 &&\quad \text{if }\ka\in\gK_S. 
\eal 
\right. 
\end{equation}  
Then 
$$ 
W_{\coW\ka}^s:=W_{\coW p}^s(\BX_\ka,g_\ka) 
\qa \ka\in\gK_0\cup\gK_\Ga, 
$$ 
and 
$$ 
\bar W_{\coW\ka}^s:=\bar W_{\coW p}^s(\BR^m\ssm\pl\BH^m,\,g_\ka) 
\qa \ka\in\gK_S, 
$$ 
where
\hb{0\leq s\leq2}. For the sake of a unified presentation,  
$$ 
\sW_{\coW\ka}^s:=  
\left\{ 
\bal 
{}
&W_\ka^s,
 &&\quad \text{if }\ka\in\gK\ssm\gK_S,\cr 
&\bar W_{\coW\ka}^s,
 &&\quad \text{if }\ka\in\gK_S. 
\eal 
\right. 
$$ 
If 
\hb{\BX\in\{\BR^m,\BH^m\}}, then 
\hb{W_{\coW p}^s(\BX):=W_{\coW p}^s(\BX,g_m)}. It is a consequence of 
\eqref{C.g} and \eqref{C.Nd} that 
\begin{equation}\label{C.Wk} 
\sW_{\coW\ka}^k\doteq\sW_{\coW p}^k(\BX_\ka) 
\qquad\text{$\gK$-unif.}, 
\npbd 
\end{equation}  
where 
\hb{{}\doteq{}}~stands for `equal except for equivalent norms'.\po  

\par 
Since 
$$ 
\sW_{\coW p}^s(\BX) 
=\bigl(\sW_{\coW p}^k(\BX),\sW_{\coW p}^{k+1}(\BX)\bigr)_{s-k,p}  
\qa k<s<k+1, 
$$ 
(cf.~\cite[Theorems VII.2.7.4 and VII.2.8.3, as well as 
(VII.3.6.3)]{Ama19a}), 
it follows from definition~\eqref{F.Ws} and from \eqref{C.Wk} that 
\begin{equation}\label{C.W} 
\sW_{\coW\ka}^s\doteq\sW_{\coW p}^s(\BX_\ka) 
\qquad\text{$\gK$-unif.} 
\end{equation}  
Due to \eqref{C.g} and \eqref{C.Nd} we get, with an analogous definition  
of~$\sBC$, 
\begin{equation}\label{C.B} 
\sBC_\ka^k 
:=\sBC^k(\BX_\ka,g_\ka)
\doteq\sBC^k(\BX_\ka):=\sBC^k(\BX_\ka,g_m) 
\qquad\text{$\gK$-unif.} 
\end{equation}  
Using this, \eqref{F.WJ}, and \eqref{F.BJ}, we infer that 
\begin{equation}\label{C.WW} 
\sW_{\coW\ka}^{s/\mf2}\doteq\sW_{\coW p}^{s/\mf2}(\BX_\ka\times J) 
\qb \sBC_\ka^{k/\mf2}\doteq\sBC^{k/\mf2}(\BX_\ka\times J) 
\qquad \text{$\gK$-unif}.
\end{equation}\po   

\par 
First we note that \eqref{U.K}, \eqref{L.a}, \eqref{M.a}, and 
\eqref{C.Nd} imply 
\begin{equation}\label{C.ak} 
a_\ka\in\sBC_\ka^{1/\mf2} 
\qquad\text{$\gK$-unif}. 
\end{equation}  
In local coordinates, 
\hb{\grad u=g^{ij}\pl_ju\,\pl/\pl x^i}. Using this, \eqref{U.g}, and 
\eqref{C.ak} we deduce that 
\begin{equation}\label{C.a} 
\|\grad_\ka a_\ka\|_{\sBC_\ka(T\BX_\ka\times J)}\leq c 
\qquad\text{$\gK$-unif.} 
\end{equation}  

\par 
Given a vector field 
\hb{Y=Y^i\pl/\pl x^i}, it holds 
\hb{\tdiv Y=\sqrt{g}^{\kern1pt-1}\pl_i\bigl(\sqrt{g}\,Y^i\bigr)}. 
By this and the above it is verified that 
\begin{equation}\label{C.Aku} 
\cA_\ka\in\cL(\sW_{\coW\ka}^{2/\mf2},\sW_{\coW\ka}^{0/\mf2}) 
\qquad\text{$\gK$-unif.} 
\end{equation}  

\par 
Now we consider Sobolev--Slobodeckii spaces on 
\hb{\pl\BH^m\simeq\BR^{m-1}}. We set 
\hb{g_\ka^\fdot:=g_{\ka\pl\BH^m}} for 
\hb{\ka\in\gK_\Ga}. Then 
\begin{equation}\label{C.diW} 
\pl_iW_{\coW\ka} 
:=W_{\coW p}^{(2-i-1/p)/\mf2}(\pl\BH^m\times J,\,g_\ka^\fdot+dt^2) 
\qa i=0,1, 
\npbd 
\end{equation}  
and 
$$ 
\pl W_{\coW\ka}:=(1-\da_\ka)\pl_0W_{\coW\ka}+\da_\ka\pl_1W_{\coW\ka} 
\qa \ka\in\gK_\Ga. 
$$ 

\par 
Suppose 
\hb{0<\sa<s<1}. Then 
\begin{equation}\label{C.BB} 
\bal 
BC^{1/\mf2}(\BR^{m-1}\times J) 
&=C\bigl(J,BC^1(\BR^{m-1})\bigr)\cap C^{1/2}\bigl(J,BC(\BR^{m-1})\bigr)\cr 
&=C\bigl(J,BC^1(\BR^{m-1})\bigr)\cap C^{1/2}\bigl(J,\BUC(\BR^{m-1})\bigr)\cr 
&\hr B\bigl(J,\BUC^s(\BR^{m-1})\bigr) 
 \cap C^{s/2}\bigl(J,\BUC(\BR^{m-1})\bigr)\cr 
&\doteq\BUC^{s/\mf2}(\BR^{m-1}\times J) 
 \hr b_\iy^{\sa/\mf2}(\BR^{m-1}\times J), 
\eal 
\end{equation}  
where the $\BUC^\rho$~are the usual H\"older spaces and $b_\iy^{\sa/\mf2}$ 
is an anisotropic little Besov space. Indeed, the first 
embedding follows from the mean value theorem and by using the localized 
H\"older norm (cf.~\cite[(VII.3.7.1)]{Ama19a}). 
For the norm equivalence we refer to definition~(VII.3.6.4) and 
Remark~VII.3.6.4. The last embedding is implied by Lemma~VII.2.2.3 and 
Theorem~VII.7.3.4. By Theorem~VII.2.7.4 in~\cite{Ama19a}, 
\begin{equation}\label{C.int} 
\bal 
{}
&b_\iy^{\sa/\mf2}(\BR^{m-1}\times J)\cr 
&\qquad{} 
 \doteq\bigl(\BUC^{2/\mf2}(\BR^{m-1}\times J),\, 
 \BUC^{0/\mf2}(\BR^{m-1}\times J)\bigr)_{\sa/2,\iy}^0. 
\eal 
\end{equation}  
We deduce from \eqref{U.g} and \eqref{C.Nd} that 
$$ 
\BUC_\ka^{s/\mf2}(\pl\BH^m\times J) 
:=\BUC^{s/\mf2}(\pl\BH^m\times J,\,g_\ka^\fdot+dt^2) 
\doteq\BUC^{s/\mf2}(\BR^{m-1}\times J) 
$$ 
$\gK_\Ga$-unif. Now it follows from \eqref{C.BB} and \eqref{C.int} that 
$$ 
BC_\ka^{1/\mf2}(\pl\BH^m\times J) 
\hr b_{\iy,\ka}^{\sa/\mf2}(\pl\BH^m\times J) 
\qquad\text{$\gK_\Ga$-unif.} 
$$ 
Since, trivially, 
\hb{\ga\in\cL\bigl(BC^{1/\mf2}(\BH^m\times J), 
   \,BC^{1/\mf2}(\pl\BH^m\times J)\bigr)}, it is now clear that 
\begin{equation}\label{C.ab} 
\ga a_\ka\in b_{\iy,\ka}^{\sa/\mf2}(\pl\BH^m\times J) 
\qquad\text{$\gK_\Ga$-unif.} 
\end{equation}  

\par 
In local coordinates 
\begin{equation}\label{C.nu} 
\nu_\ka=\frac{g_\ka^{1i}}{\sqrt{g_\ka^{11}}}\,\frac\pl{\pl x^i}. 
\end{equation}  
Hence 
\hb{\da_\ka\cB_\ka u=\da_\ka b_\ka^i\ga\pl_iu 
   =\da_\ka\ga a_\ka\nu_\ka^i\ga\pl_iu_\ka}, 
where it follows from \eqref{U.g}, 
\eqref{M.aa}, and 
\hb{\|a_\ka\|_\iy\leq c} that 
\begin{equation}\label{C.b1} 
1/c\leq b_\ka^1=\ga a_\ka\sqrt{g_\ka^{11}}\leq c 
\end{equation}  
and, from \eqref{C.ab}, 
$$ 
\|b_\ka^i\|_{b_{\iy,\ka}^{\sa/\mf2}(\pl\BH^m\times J)}\leq c 
\qa 1\leq i\leq m, 
$$ 
for 
\hb{\ka\in\gK_\Ga}. Thus it is a consequence of \eqref{C.ab}, 
\eqref{C.b1}, and the boundary operator retraction theorem 
\cite[Theorem~VIII.2.2.1]{Ama19a} 
that\addtocounter{footnote}{1}\footnotetext{An operator 
\hb{r\in\cL\EF} is a retraction if it has a continuous right inverse, 
a~coretraction~$r^c$. Then $(r,r^c)$ is an 
\emph{\hbox{r-c} pair for}~$\EF$.} 
\begin{equation}\label{C.RB} 
\addtocounter{footnote}{-1}
\cB_\ka\text{ is a $\gK_\Ga$-uniform retraction\footnotemark\  
from $W_{\coW\ka}^{2/\mf2}$ onto }\pl W_{\coW\ka}. 
\end{equation}  
Clearly, `\hbox{$\gK_\Ga$-uniform}' means that there exists a 
coretraction~$\cB_\ka^c$ such that $\|\cB_\ka\|$ and~$\|\cB_\ka^c\|$ are 
\hbox{$\gK_\Ga$-uniformly} bounded. 

\par 
Obviously, 
\begin{equation}\label{C.WSJ} 
\pl_SW_{\coW\ka}:=\pl_0W_{\coW\ka}\oplus\pl_1W_{\coW\ka} 
\qa \ka\in\gK_S. 
\end{equation}  
If we replace in the preceding arguments the boundary 
operator retraction argument by Theorem~VIII.2.3.3 of~\cite{Ama19a}, 
we find that 
\begin{equation}\label{C.Rc} 
\cC_\ka\text{ is a $\gK_S$-uniform retraction 
from $\bar W_{\coW\ka}^{2/\mf2}$ onto }\pl_S W_{\coW\ka}. 
\end{equation}  

\par 
It follows from \eqref{C.W} that 
\begin{equation}\label{C.dW} 
\ga_0\sW_{\coW\ka}\doteq\ga_0\sW_{\coW p}(\BX_\ka) 
\qa \ka\in\gK. 
\end{equation}  
The anisotropic trace theorem (\cite[Corollary~VII.4.6.2, 
Theorems VIII.1.2.9 and~VIII.1.3.1]{Ama19a}) implies that 
\begin{equation}\label{C.g0} 
\ga_0\in\cL\bigl(W_{\coW p}^{2/\mf2}(\BX\times J),B_p^{2-2/p}(\BX)\bigr) 
\qa \BX\in\{\BR^m,\BH^m\}, 
\end{equation}  
is a retraction. Using Theorems VII.2.7.4 and~VII.2.8.3, 
definition~VII.3.6.3 and Remark~VII.3.6.4 of~\cite{Ama19a}, we get 
\begin{equation}\label{C.BW} 
B_p^{2-2/p}(\BX)\doteq W_{\coW p}^{2-2/p}(\BX). 
\end{equation}  
Now we infer from \eqref{C.WW}, \hbox{\eqref{C.dW}--\eqref{C.BW}},   
and \eqref{C.W} that 
\begin{equation}\label{C.0} 
\ga_0\text{ is a $\gK$-uniform retraction 
from $\sW_{\coW\ka}^{2/\mf2}$ onto }\ga_0\sW_{\coW\ka}. 
\end{equation}  
\section{Maximal Regularity}\label{sec-MR} 
First we rewrite $\AB$ in terms of covariant derivatives. 
For this we define 
$$ 
a^\nat\in\bar{BC}^{1/\mf2}\bigl(T_0^2(M\ssm S)\times J\bigr) 
$$ 
in local coordinates by 
$$ 
a^\nat:=ag^{ij}\frac\pl{\pl x^i}\otimes\frac\pl{\pl x^j}. 
$$ 
Then we get 
\begin{equation}\label{MR.A} 
\cA u=-a\Da u-(\grad a\sn\grad u) 
=-a^\nat\btdot\na^2u-(\grad a)\btdot\na u, 
\end{equation}  
where $\Da$~is the Laplace--Beltrami operator 
(e.g.,~\cite[(2.4.10)]{Tay96a}). Consequently, 
\begin{equation}\label{MR.sA} 
\gss\cA(q,t,\xi)=a(q,t)\,|\xi|^2_{g^*(q)} 
\qa q\in M\ssm S 
\qb \xi\in T_q^*(M\ssm S) 
\qb t\in J. 
\end{equation}  

\par 
For the boundary operator we find
\begin{equation}\label{MR.B} 
\cB^1u=\ga a(\nu^\flat\sn\ga\na u)_{g^*}. 
\end{equation}  
Hence 
\begin{equation}\label{MR.sB} 
\gss\cB^1u(q,t,\xi)=a(q,t)\bigl(\nu^\flat(q)\bsn\xi\bigr)_{g^*(q)} 
\qa q\in\Ga 
\qb \xi\in T_q^*M 
\qb t\in J. 
\end{equation}  
Clearly, these formulas apply to any oriented Riemannian manifold, thus to 
$(\pl\BH^m,g_\ka)$, 
\ \hb{\ka\in\gK_\Ga}. 

\par 
It follows from \eqref{M.aa} and \eqref{MR.sA} that 
$$ 
\gss\cA_\ka(x,t,\xi) 
=a_\ka(x,t)\,|\xi|_{g_\ka^*(x)}^2\geq\ul{\al}\,|\xi|_{g_\ka^*(x)}^2 
$$ 
for 
\,\hb{x\in\BX_\ka}, 
\,\hb{\xi\in T_x^*\BX_\ka}, 
\,\hb{t\in J}, and   
\hb{\ka\in\gK}. Hence 
\begin{equation}\label{MR.ne} 
\cA_\ka\text{ is uniformly normally elliptic, unif.\ w.r.t.\ }\ka\in\gK 
\text{ and }t\in J.  
\end{equation}  

\par 
We begin with the full-space problem.
\begin{proposition}\label{pro-MR.1} 
It holds 
$$ 
(\pl_t+\cA_\ka,\,\ga_0) 
\in\Lis(W_{\coW\ka}^{2/\mf2},\,W_{\coW\ka}^{0/\mf2}\times\ga_0W_{\coW\ka}) 
\qquad\text{$\gK_0$-unif.}, 
$$ 
that is, 
$$ 
\|(\pl_t+\cA_\ka,\,\ga_0)\| 
+\|(\pl_t+\cA_\ka,\,\ga_0)^{-1}\|\leq c 
\qa \ka\in\gK_0. 
$$ 
\end{proposition}
\begin{proof} 
It is obvious from \eqref{C.Aku} and \eqref{C.0} that 
$$ 
(\pl_t+\cA_\ka,\,\ga_0) 
\in\cL(W_{\coW\ka}^{2/\mf2},\,W_{\coW\ka}^{0/\mf2}\times\ga_0 W_{\coW\ka}) 
\qquad\text{$\gK_0$-unif.}
$$ 
Due to \eqref{MR.ne}, the assertion now follows from Corollary~9.7 
in~\cite{AHS94a} and Theorem~III.4.10.8 in~\cite{Ama95a} and (the proof of) 
Theorem~7.1 in~\cite{Ama04a}. (See \cite{AmaVolIII21a} for a different 
demonstration.) 
\end{proof} 
Next we study the case where 
\hb{\ka\in\gK_\Ga}. For this we first establish the validity of the 
uniform LS condition. Henceforth, it is always assumed that 
\begin{equation}\label{MR.Z} 
\za=(x,\xi,\lda)\text{ with }x\in\pl\BH^m\text{ and }(\xi,\lda)\in\BB(x). 
\end{equation}  
We fix any 
\hb{t\in J} and omit it from the notation. The reader will easily check 
that all estimates are uniform with respect to 
\hb{t\in J}. From \eqref{MR.sA} we see that the first equation in 
\eqref{E.Aa} has the form  
\begin{equation}\label{MR.vR} 
\ddot v=\rho_\ka^2(\za)v\text{ on }\BR, 
\end{equation}  
where 
\begin{equation}\label{MR.rh} 
\rho_\ka(\za) 
:=\sqrt{\frac\lda{a_\ka(x)}+|\xi|_{g_\ka^*(x)}^2}\in\BC 
\npbd 
\end{equation}  
with the principal value of the square root.\po 

\par 
Suppose 
\hb{|\xi|_{g_\ka^*(x)}^2\leq1/2}. Then 
\hb{\za\in\BB(x)} implies 
\hb{\rho_\ka^2(\za)\geq1/2a_\ka(x)}. Otherwise, 
\hb{\rho_\ka^2(\za)\geq1/2}. Thus, since 
\hb{\|a_\ka\|_\iy\leq c}, we find a 
\hb{\ba>0} such that 
\begin{equation}\label{MR.Rr} 
\Re\rho_\ka(\za)\geq\ba 
\qa \ka\in\gK_\Ga. 
\end{equation}  
From 
\hb{a_\ka\geq\ul{\al}} we infer that 
\hb{|\rho_\ka(\za)|\leq c} for 
\hb{\ka\in\gK_\Ga}. Set 
\begin{equation}\label{MR.v} 
v_\ka(\za)(s):=e^{-\rho_\ka(\za)s} 
\qa s\geq0. 
\npbd 
\end{equation}  
Then $\BC v_\ka(\za)$~is the subspace of $C_0\RC$ of decaying solutions 
of~\eqref{MR.vR}.\po 

\par 
Let 
\hb{\ka\in\gK_{\Ga_0}} so that 
\hb{\cB_\ka=\ga}, the Dirichlet operator on~$\pl\BH^m$. Then 
(recall~\eqref{E.v}), 
\,\hb{R_\ka(\za)\eta=\eta v_\ka(\za)}. Thus, by~\eqref{MR.Rr}, 
$$ 
\|R_\ka(\za)\|_{\cL(\BC,C_0(\BR_+,\BC))}\leq 1 
\qa \ka\in\gK_{\Ga_0}. 
$$ 

\par 
Now assume 
\hb{\ka\in\gK_{\Ga_1}}. Then we see from \eqref{MR.B} and \eqref{MR.v} that 
\begin{equation}\label{MR.B0} 
\cB_\ka(\pl;\za)v_\ka(\za)(0)=-\imi a_\ka(x)\rho_\ka(\za). 
\end{equation}  
Consequently, 
$$ 
\|R_\ka(\za)\|_{\cL(\BC,C_0(\BR_+,\BC))} 
=\frac1{a_\ka(x)\,|\rho_\ka(\za)|} 
\leq\frac1{\ul{\al}\ba} 
\qa \ka\in\gK_{\Ga_1}. 
$$ 
This proves that $(\cA_\ka,\cB_\ka)$ satisfies the uniform 
parameter-dependent LS condition, 
unif.\ w.r.t.\ 
\hb{\ka\in\gK_\Ga} and 
\hb{t\in J}.   
\begin{proposition}\label{pro-MR.2} 
It holds 
$$ 
\bigl(\pl_t+\cA_\ka,\,(\cB_\ka,\ga_0)\bigr) 
\in\Lis\bigl(W_{\coW\ka}^{2/\mf2}, 
\,W_{\coW\ka}^{0/\mf2} 
\oplus[\pl W_{\coW\ka}\oplus\ga_0W_{\coW\ka}]_{\cB_\ka}\bigr) 
\qquad\text{$\gK_\Ga$-unif.} 
$$ 
\end{proposition}
\begin{proof} 
We deduce from \eqref{C.RB}, \eqref{C.0}, and 
\cite[Example~VIII.1.8.6]{Ama19a} that 
$$ 
[\pl W_{\coW\ka}\oplus\ga_0W_{\coW\ka}]_{\cB_\ka} 
$$ 
is a well-defined closed linear subspace of 
\hb{\pl W_{\coW\ka}\oplus\ga_0W_{\coW\ka}} and, using 
also~\eqref{C.Aku}, 
$$ 
\bigl(\pl_t+\cA_\ka,\,(\cB_\ka,\ga_0)\bigr) 
\in\cL\bigl(W_{\coW\ka}^{2/\mf2}, 
\,W_{\coW\ka}^{0/\mf2} 
\oplus[\pl W_{\coW\ka}\oplus\ga_0W_{\coW\ka}]_{\cB_\ka}\bigr) 
\qquad\text{$\gK_\Ga$-unif.} 
\npbd 
$$ 
The uniform LS~condition implies now the remaining assertions. For this we 
refer to~\cite{AmaVolIII21a}. 
\end{proof} 
Nonhomogeneous linear parabolic boundary value problems (of arbitrary order 
and in a Banach-space-valued setting) on Euclidean domains have been studied 
in~\cite{DHP07a}. It follows, in particular from Proposition~6.4 therein, 
that the isomorphism assertion is true for each 
\hb{\ka\in\gK_\Ga}. However, it is not obvious whether the 
\hbox{$\gK_\Ga$-uniformity} statement does also follow from the results 
in~\cite{DHP07a}. For this one would have to check carefully the dependence 
of all relevant estimates on the various parameters involved, 
which would be no easy task. (The same observation applies to 
Proposition~\ref{pro-MR.1}.) In~\cite{AmaVolIII21a} we present an 
alternative proof which takes care of the needed uniform estimates. 

\par 
Lastly, we assume that 
\hb{\ka\in\gK_S}. We set, once more suppressing a fixed 
\hb{t\in J},  
$$ 
a_\ka^1(x):=a_\ka(x) 
\qb a_\ka^2(x):=a_\ka(-x) 
\qa x\in\BH^m, 
$$ 
and 
$$ 
\gsa_\ka:=\diag[a_\ka^1,a_\ka^2]\sco\BH^m\ra\BC^{2\times 2}. 
$$ 
Then 
$$ 
\gA_\ka\gsu:=-\tdiv_\ka(\gsa_\ka\grad_\ka\gsu) 
\qa \gsu=(u^1,u^2). 
$$ 
Furthermore, 
\hb{\gB_\ka=(\gB_\ka^0,\gB_\ka^1)}, where 
\begin{equation}\label{MR.B2} 
\bal 
\gB_\ka^0\gsu 
&:=\ga u^1-\ga u^2,\cr 
\gB_\ka^1\gsu 
&:=\bigl(\nu_\ka\bsn\ga(a_\ka^1\grad_\ka u^1+a_\ka^2\grad_\ka u^2) 
 \bigr)_{g_\ka} 
\eal 
\npbd 
\end{equation}  
on~$\pl\BH^m$.\po 

\par 
Clearly, 
$$ 
\sa\bigl(\gsa(x)\bigr)\is[\Re z\geq\ul{\al}] 
\qa x\in\BH^m. 
\npbd 
$$ 
Thus $\gA_\ka$~is uniformly normally elliptic on~$\BH^m$, 
\hbox{$\gK_S$-unif.}\po  

\par 
We define~$\rho_\ka^i$, 
\,\hb{1=1,2}, by replacing~$a_\ka$ in \eqref{MR.rh} by~$a_\ka^i$ and 
introduce~$v_\ka^i$ by changing~$\rho_\ka$ in\eqref{MR.v} to~$\rho_\ka^i$. 
Then 
$$ 
\BC v_\ka^1\oplus\BC v_\ka^2 
$$ 
is the subspace of $C_0(\BR_+,\BC^2)$ of decaying solutions of 
$$ 
\bigl(\lda+\gss\gA_\ka(x,\,\xi+\imi\nu_\ka(x)\pl)\bigr)\gsv=0 
\qa x\in\pl\BH^m 
\qb \gsv=(v^1,v^2). 
$$ 
From \eqref{MR.B0} and \eqref{MR.B2} we see that the initial conditions 
in~\eqref{E.LS} are in the present case 
$$ 
\bal 
v^1(0)-v^2(0) 
&=\eta^1,\cr 
a_\ka^1(x)\rho_\ka^1(x)
v^1(0)+a_\ka^2(x)\rho_\ka^2(x)v^2(0) 
&=\imi\eta^2. 
\eal 
$$ 
Omitting 
\hb{x\in\pl\BH^m}, the solution of this system is 
$$ 
\bal 
v_\ka^1(0) 
&=v_\ka^2(0)+\eta^1,\cr 
v_\ka^2(0) 
&=\frac1{a_\ka^1\rho_\ka^1+a_\ka^2\rho_\ka^2} 
 (\imi\eta^2-a_\ka^1\rho_\ka^1\eta^1). 
\eal 
$$ 
From this, the uniform boundedness of~$a_\ka$, and 
\hb{\Re(a_\ka^1\rho_\ka^1+a_\ka^2\rho_\ka^2)\geq1/\ul{\al}\ba} 
it follows that $(\gA_\ka,\gB_\ka)$ satisfies the uniform 
parameter-dependent LS condition, unif.\ w.r.t.\ 
\hb{\ka\in\gK_S} and 
\hb{t\in J}.
\begin{proposition}\label{pro-MR.3} 
It holds 
$$ 
\bigl(\pl_t+\cA_\ka,\ (\cC_\ka,\ga_0)\bigr) 
\in\Lis\bigl(\bar W_{\coW\ka}^{2/\mf2}, 
\,\bar W_{\coW\ka}^{0/\mf2} 
\oplus[\pl_S\bar W_{\coW\ka}\oplus\ga_0\bar W_{\coW\ka}]_{\cC_\ka}\bigr) 
\npbd 
$$ 
unif.\ w.r.t.\ 
\hb{\ka\in\gK_S} and 
\hb{t\in J}. 
\end{proposition}
\begin{proof} 
Set 
\hb{\gsu(x):=\bigl(u(x),u(-x)\bigr)} for 
\hb{x\in\BH^m} and 
\hb{\bar\gW_\ka^s:=\bar W_\ka^s\oplus\bar W_\ka^s} etc. Then the assertion 
is true iff 
$$ 
\bigl(\pl_t+\gA_\ka,\ (\gB_\ka,\ga_0)\bigr) 
\in\Lis\bigl(\bar\gW_\ka^{2/\mf2}, 
\,\bar\gW_\ka^{0/\mf2} 
\oplus[\pl_S\bar\gW_\ka\oplus\ga_0\bar\gW_\ka]_{\gB_\ka}\bigr) 
\qquad\text{$\gK_S$-unif.} 
$$ 
By the preceding considerations, the proof of Proposition~\ref{pro-MR.2} 
applies verbatim to the system for~$\gsu$. This proves the claim. 
\end{proof} 
\section{Localizations}\label{sec-LT} 
We presuppose \eqref{L.ass} and fix an \hbox{$S$-adapted} atlas for~$M$ 
with 
\hb{\diam(U_{\coU\ka})<\ve/2} for 
\hb{\ka\in\gK_S}. Then 
$$ 
\gN(\ka):=\{\,\wk\in\gK\ ;\ U_{\coU\ka}\cap U_\wk\neq\es\,\} 
\qa \ka\in\gK, 
$$ 
\hb{\gN_\Ga(\ka):=\gN(\ka)\cap\gK_\Ga}, 
and \hb{\gN_S(\ka):=\gN(\ka)\cap\gK_S}. By the finite multiplicity of~$\gK$, 
\begin{equation}\label{LT.k} 
\card\bigl(\gN(\ka)\bigr)\leq c 
\qa \ka\in\gK. 
\end{equation}  
We set for 
\hb{\ka\in\gK} and 
\hb{\wk\in\gN(\ka)} 
$$ 
S_\kwk u:=\ka_*\wk^*u=u\circ(\wk\circ\ka^{-1}) 
\qa u\in\sW_\wk^{0/\mf2}. 
$$ 
It follows from \eqref{U.K}(ii) that, given 
\hb{s\in[0,2]}, 
\begin{equation}\label{LT.S} 
S_\kwk\in\cL(\sW_\wk^{s/\mf2},\sW_{\coW\ka}^{s/\mf2}) 
\qquad\text{$\gK$-unif.}
\end{equation}  

\par 
Interpreting~$\gK$ as an index set, we put 
\begin{equation}\label{LT.Wp} 
\mfsW^{s/\mf2}:=\prod_{\ka\in\gK}\sW_{\coW\ka}^{s/\mf2}, 
\end{equation}  
endowed with the product topology. For 
\hb{\al\in\{0,\Ga\}} we set 
$$ 
\mfW^{s/\mf2}[\al]:=\prod_{\ka\in\gK_\al}W_{\coW\ka}^{s/\mf2} 
\qb \bar\mfW^{s/\mf2}:=\prod_{\ka\in\gK_S}\bar W_{\coW\ka}^{s/\mf2}. 
$$ 
Since $\gK$~is the disjoint union of $\gK_0$, $\gK_\Ga$, and~$\gK_S$, 
\begin{equation}\label{P.Ws} 
\mfsW^{s/\mf2} 
=\mfW^{s/\mf2}[0]\oplus\mfW^{s/\mf2}[\Ga]\oplus\bar\mfW^{s/\mf2}. 
\end{equation}  
A similar definition and direct sum decomposition applies 
to $\ga_0\mfsW$. We also set 
$$ 
\pl\mfW:=\prod_{\ka\in\gK_\Ga}\pl W_{\coW\ka} 
\qb \pl_S\mfW:=\prod_{\ka\in\gK_S}\pl_S\bar W_{\coW\ka}. 
$$ 

\par 
We define linear operators $\cR$ 
and~$\cR^c$ by 
\begin{equation}\label{LT.RR} 
\cR\mfu:=\sum_\ka\pi_\ka\ka^*u_\ka 
\qb \cR^cu:=\bigl(\ka_*(\pi_\ka u)\bigr)_{\ka\in\gK} 
\end{equation}  
for 
\hb{\mfu=(u_\ka)\in\mfsW^{0/\mf2}} and 
\hb{u\in L_1\bigl(J,L_{1,\loc}(M\ssm S)\bigr)}, respectively. 
The sum is locally finite 
and $\pi_\ka$~is identified with the multiplication operator 
\hb{v\mt\pi_\ka v}. 

\par 
We want to evaluate 
\hb{\cA\circ\cR\mfu} for 
\hb{\mfu\in\mfsW^{2/\mf2}}. Observe  
$$ 
\cA(\pi_\ka u)=\pi_\ka\cA u+[\cA,\pi_\ka]u 
\qa u\in W_{\coW p}^{2/\mf2}, 
$$ 
the commutator being given by 
\begin{equation}\label{LT.Ap} 
[\cA,\pi_\ka]u 
=-(\grad\pi_\ka\sn a\grad u)-\tdiv(au\grad\pi_\ka). 
\end{equation}  
Thus we get from \eqref{LT.RR} 
\begin{equation}\label{LT.AR} 
\cA\cR\mfu 
=\sum_\ka\cA(\pi_\ka\ka^*u_\ka) 
=\sum_\ka\pi_\ka\cA(\ka^*u_\ka)+\sum_\ka[\cA,\pi_\ka]\ka^*u_\ka.    
\end{equation}  
By 
\hb{\cA(\ka^*u_\ka)=\ka^*\cA_\ka u_\ka}, the first sum 
equals~$\cR\mfsA\mfu$, where 
\hb{\mfsA:=\diag[\cA_\ka]}. Using 
\hb{1=\sum_\wk\pi_\wk^2}, we find 
\begin{equation}\label{LT.AR2} 
\bal 
{}
[\cA,\pi_\ka]\ka^*u_\ka 
&=\sum_\wk\pi_\wk\pi_\wk[\cA,\pi_\ka]\ka^*u_\ka\cr 
&=\sum_\wk\pi_\wk\wk^* 
 \bigl((\wk_*\pi_\wk)\wk_*[\cA,\pi_\ka]\wk^* 
 (\wk_*\ka^*)u_\ka\bigr)\cr 
&=\sum_{\wk\in\gN(\ka)}\pi_\wk\wk^* 
 \bigl((\wk_*\pi_\wk)[\cA_\wk,S_\wkk(\ka_*\pi_\ka)]S_\wkk u_\ka\bigr).
\eal 
\end{equation}  
Set 
$$ 
\cA_\kwk:=(\ka_*\pi_\ka) 
\bigl[\cA_\ka,S_\kwk(\wk_*\pi_\wk)\bigr]S_\kwk\chi. 
$$ 
Then \eqref{LT.S}, \eqref{LT.Ap}, \eqref{M.LS}, \eqref{C.ak}, and 
\hb{\ka_*\pi_\ka=(\ka_*\pi_\ka)\chi} 
imply 
\begin{equation}\label{LT.Akk} 
\|\cA_\kwk v\|_{\sW_{\coW\ka}^{0/\mf2}} 
\leq c\,\|\chi v\|_{\sW_{\coW\wk}^{1/\mf2}} 
\qa \wk\in\gN(\ka) 
\qb \ka\in\gK. 
\end{equation}  
We define 
\hb{\mfsA_\ka^0\sco\mfsW^{1/\mf2}\ra\sW_{\coW\ka}^{0/\mf2}} by 
$$ 
\mfsA_\ka^0\mfu:=\sum_{\wk\in\gN(\ka)}\cA_\kwk u_\wk 
\qa \mfu\in\mfsW^{1/\mf2} 
\qb \ka\in\gK. 
$$ 
Then we deduce from \eqref{LT.k} and \eqref{LT.Akk} that 
\begin{equation}\label{LT.Ak} 
\|\mfsA_\ka^0\mfu\|_{\sW_{\coW\ka}^{0/\mf2}} 
\leq c\sum_{\wk\in\gN(\ka)}\|\chi u_\wk\|_{\mfsW_\wk^{1/\mf2}} 
\qa u\in\mfsW^{1/\mf2} 
\qb \ka\in\gK. 
\npbd 
\end{equation}  
Moreover, 
\hb{\mfsA^0:=(\mfsA_\ka^0)_{\ka\in\gK}}.\po  

\par 
Now we sum \eqref{LT.AR2} over 
\hb{\ka\in\gK} and interchange the order of summation to find 
that the second sum in \eqref{LT.AR} equals~$\cR\mfsA^0\mfu$. 
Thus, in total,  
\begin{equation}\label{LT.RAA} 
\cA\cR=\cR(\mfsA+\mfsA^0).  
\end{equation}  

\par 
Similar considerations lead to 
\begin{equation}\label{LT.AAR} 
\cR^c\cA=(\mfsA+\wt{\mfsA}\vph{\sA}^0)\cR^c. 
\end{equation}  
Here 
\hb{\wt{\mfsA}\vph{\sA}^0=(\wt{\mfsA}\vph{\sA}_\ka^0)_{\ka\in\gK}} with 
\hb{\wt{\mfsA}\vph{\sA}_\ka^0\in\cL(\mfsW^{1/\mf2},\sW_{\coW\ka}^{0/\mf2})} 
satisfying  
\begin{equation}\label{LT.At} 
\|\wt{\mfsA}\vph{\sA}_\ka^0\mfu\|_{\sW_{\coW\ka}^{0/\mf2}} 
\leq c\sum_{\wk\in\gN(\ka)}\|\chi u_\wk\|_{\mfsW_\wk^{1/\mf2}} 
\qa \mfu\in\mfsW^{1/\mf2} 
\qb \ka\in\gK. 
\end{equation}  

\par 
We turn to~$\Ga$ and define 
\hb{\ka^\fdot:=\ga\ka} for 
\hb{\ka\in\gK_\Ga}. Then 
\begin{equation}\label{LT.Rga} 
\cR_\Ga\mfu:=\sum_{\ka\in\gK_\Ga}\ga\pi_\ka(\ka^\fdot)^*u_\ka 
\qa \mfu=(u_\ka)\in\pl\mfW, 
\end{equation}  
and 
\begin{equation}\label{LT.Rgc} 
\cR_\Ga^cu:=\bigl(\ka_*^\fdot\ga(\pi_\ka u)\bigr)_{\ka\in\gK_\Ga} 
\qa u\in W_{\coW p}^{2/\mf2}\bigl((M\ssm S)\times J\bigr). 
\end{equation}  

\par 
Observe that 
$$ 
\cB(\pi_\ka u)=(\ga\pi_\ka)\cB u+[\cB,\pi_\ka]u 
$$ 
and 
$$ 
[\cB,\pi_\ka]u=\da\bigl(\nu\bsn\ga(a\grad\pi_\ka)\bigr)_{g^\ifdot}\ga u  
\qa u\in W_{\coW p}^{2/\mf2}\bigl((M\ssm S)\times J\bigr). 
$$ 
Similarly as above, we find 
\begin{equation}\label{LT.RBB} 
\cB\cR=\cR_\Ga(\mfB+\mfB^0), 
\end{equation}  
where 
\hb{\mfB:=\diag[\cB_\ka]} and 
\hb{\mfB^0=(\mfB_\ka^0)_{\ka\in\gK_\Ga}} with 
\hb{\mfB_\ka^0\sco\mfW^{1/\mf2}[\Ga]\ra\{0\}\oplus\pl_1W_{\coW\ka}} 
satisfying 
\begin{equation}\label{LT.Bk} 
\|\mfB_\ka^0\mfu\|_{\pl_1W_{\coW\ka}} 
\leq c\sum_{\wk\in\gN_\Ga(\ka)}\|\chi u_\wk\|_{W_\wk^{1/\mf2}} 
\qa \ka\in\gK_\Ga. 
\end{equation}  
Analogously, 
\begin{equation}\label{LT.BBR} 
\cR_\Ga^c\cB=(\mfB+\wt{\mfB}\vph{\mfB}^0)\cR^c, 
\end{equation}  
where 
\hb{\wt{\mfB}\vph{\mfB}^0=(\wt{\mfB}\vph{\mfB}_\ka^0)} with 
\hb{\wt{\mfB}\vph{\mfB}_\ka^0 
   \sco\mfW^{1/\mf2}[\Ga]\ra\{0\}\oplus\pl_1W_{\coW\ka}} is such that 
\begin{equation}\label{LT.Bk0} 
\|\wt{\mfB}\vph{\mfB}_\ka^0\mfu\|_{\pl_1W_{\coW\ka}} 
\leq c\sum_{\wk\in\gN_\Ga(\ka)}\|\chi u_\wk\|_{W_\wk^{1/\mf2}} 
\qa \ka\in\gK_\Ga. 
\end{equation}  
Concerning the transmission interface~$S$, we define 
$\cR_S$ and~$\cR_S^c$ analogously to \eqref{LT.Rga} and \eqref{LT.Rgc}. 
Observe that 
$$ 
[\cC,\pi_\ka]u 
=\bigl(0,\eea{(\nu_S\sn a\grad\pi_\ka)u}\bigr). 
$$ 
From this it is now clear that 
\begin{equation}\label{LT.CCR} 
\cC\cR=\cR_S(\mfC+\mfC^0) 
\qb \cR_S^c\cC=(\mfC+\wt{\mfC}\vph{\mfC}^0)\cR^c, 
\end{equation}  
where 
\hb{\mfC:=\diag[\cC_\ka]}, 
\ \hb{\mfC^0=(\mfC_\ka^0)}, and 
\hb{\wt{\mfC}\vph{\mfC}^0=(\wt{\mfC}\vph{\mfC}_\ka^0)} with 
\begin{equation}\label{LT.Ck} 
\bal 
\|C_\ka^0\mfu\|_{\pl_1W_{\coW\ka}} 
&\leq c\sum_{\wk\in\gN_S(\ka)}\|\chi u_\wk\|_{\bar W_\wk^{1/\mf2}},\cr 
\|\wt{C}\vph{C}_\ka^0\mfu\|_{\pl_1W_{\coW\ka}} 
&\leq c\sum_{\wk\in\gN_S(\ka)}\|\chi u_\wk\|_{\bar W_\wk^{1/\mf2}} 
\eal 
\npbd 
\end{equation}  
for 
\hb{\ka\in\gK_S} and 
\hb{\mfu\in\bar\mfW^{2/\mf2}}.\po  

\par 
The following consequences of the preceding results are needed 
to establish Theorem~\ref{thm-L.MR}.
\begin{lemma}\label{lem-LT.ch} 
Fix 
\hb{s\in(1,2)} and 
\hb{q>p} with 
\hb{1/p-1/q<(2-s)/(m+2)}.\break 
Then 
$$ 
(v\mt\chi v) 
\in\cL\bigl(\sW_{\coW\ka}^{2/\mf2},L_q(J,\sW_{\coW\ka}^s) 
\cap W_{\coW q}^{s/2}(J,\sW_{\coW\ka}^0)\bigr) 
\qquad\text{$\gK$-unif.} 
$$ 
\end{lemma}
\begin{proof} 
(1) 
Set 
\hb{\BX:=\BR^m} and 
\hb{s_1:=s+(m+2)(1/p-1/q)<2}. By (VII.3.6.3) and Example~VII.3.6.5 
of~\cite{Ama19a} 
\begin{equation}\label{LT.Ws2} 
W_{\coW q}^{s/\mf2} 
\doteq L_q(J,W_{\coW q}^s)\cap W_{\coW q}^{s/2}(J,L_q), 
\end{equation}  
where 
\hb{W_{\coW q}^{s/\mf2}=W_{\coW q}^{s/\mf2}(\BX\times J)} etc. Hence, see 
\cite[Theorem VII.2.2.4(iv)]{Ama19a}, 
$$ 
W_{\coW p}^{2/\mf2}\hr W_{\coW p}^{s_1/\mf2} 
\hr L_q(J,W_{\coW q}^s)\cap W_{\coW q}^{s/2}(J,L_q).
$$ 
Consequently, invoking \eqref{C.W}, 
$$ 
W_{\coW\ka}^{2/\mf2} 
\hr L_q(J,W_{\coW q,\ka}^s)\cap W_{\coW q}^{s/2}(J,L_{q,\ka})
\qquad\text{$\gK_0$-unif.} 
$$ 

\par 
(2) 
Since 
\hb{\supp(\chi)\is(-1,1)^m}, H\"older's inequality implies 
$$ 
\|\chi v\|_{W_{\coW p}^k}\leq c\,\|\chi v\|_{W_{\coW q}^k} 
\qa v\in W_{\coW q}^2 
\qb k=0,1,2. 
$$ 
Hence, by interpolating and then using \eqref{C.W} once more,  
$$ 
\|\chi v\|_{W_{\coW\ka}^\sa}\leq c\,\|\chi v\|_{W_{\coW q,\ka}^\sa} 
\qquad\text{$\gK_0$-unif.} 
\qb \sa\in\{0,s\}. 
$$ 
From this we get 
$$ 
\|\chi u\|_{L_q(J,W_{\coW\ka}^s)} 
\leq c\,\|\chi u\|_{L_q(J,W_{\coW q,\ka}^s)} 
\qquad\text{$\gK_0$-unif.} 
$$ 
and 
$$ 
\|\chi u\|_{W_{\coW q}^{s/2}(J,W_{\coW\ka}^0)} 
\leq c\,\|\chi u\|_{W_{\coW q}^{s/2}(J,W_{\coW q,\ka}^0)} 
\qquad\text{$\gK_0$-unif.} 
\npbd 
$$ 
Now the assertion follows in this case from step~(1). 

\par 
(3) 
The preceding arguments also hold if we replace 
\hb{\BX=\BR^m} by 
\hb{\BX=\BH^m}. Then (see Remark~\ref{rem-F.H}) it also applies to  
\hb{\BX=\BR^m\ssm\pl\BH^m}. Thus the lemma is proved. 
\end{proof} 
Given a function space~$\gF$ defined on~$J$, we write~$\gF(\tau)$ for its 
restriction to~$J_\tau$, 
\ \hb{0<\tau\leq T}. 
\begin{lemma}\label{lem-LT.A0} 
Let 
\hb{\wh{\mfsA}_\ka\in\{\mfsA_\ka^0,\wt{\mfsA}_\ka\}}. There exists  
\hb{\ve>0} such that 
$$ 
\|\wh{\mfsA}_\ka\mfu\|_{\sW_{\coW\ka}^{0/\mf2}(\tau)} 
\leq c\coc\tau^\ve\sum_{\wk\in\gN(\ka)}\|u_\wk\|_{\sW_\ka^{2/\mf2}(\tau)} 
\qa \mfu\in\mfsW^{2/\mf2}, 
\npbd 
$$ 
unif.\ w.r.t.\ 
\hb{\ka\in\gK} and 
\hb{0<\tau\leq T}. 
\end{lemma}
\begin{proof} 
Fix $s$ and~$q$ as in Lemma~\ref{lem-LT.ch} and set 
\hb{\ve:=1/p-1/q}. We get from H\"older's inequality, Lemma~\ref{lem-LT.ch}, 
\eqref{LT.Ak}, and \eqref{LT.At} 
$$ 
\bal 
\|\wh{\mfsA}_\ka\mfu\|_{\sW_{\coW\ka}^{0/\mf2}(\tau)} 
&=\|\wh{\mfsA}_\ka\mfu\|_{L_p(J_\tau,\sW_{\coW\ka}^0)} 
 \leq\tau^\ve\,\|\wh{\mfsA}_\ka\mfu\|_{L_q(J_\tau,\sW_{\coW\ka}^0)}\cr 
&\leq c\coc\tau^\ve\sum_{\wk\in\gN(\ka)} 
 \|\chi u_\wk\|_{L_q\cap\sW_\wk^{s/2}(\tau)} 
 \leq c\coc\tau^\ve\sum_{\wk\in\gN(\ka)}\|u_\wk\|_{\sW_\ka^{2/\mf2}(\tau)} 
\eal 
$$ 
for 
\hb{0<\tau\leq T}, \hbox{$\gK$-unif.}, where 
\hb{\Vsdot_{L_q\cap\sW_{\coW\ka}^{s/2}}} is the norm in the image space 
occurring in Lemma~\ref{lem-LT.ch}. 
\end{proof} 
The next lemma provides analogous estimates for the boundary and 
transmission operators. 
\begin{lemma}\label{lem-LT.BC} 
Let 
\hb{\wh{\mfB}_\ka\in\{\mfB_\ka^0,\wt{\mfB}\vph{\mfB}_\ka^0\}} and 
\hb{\wh{\mfC}_\ka\in\{\mfC_\ka^0,\wt{\mfC}\vph{\mfC}_\ka^0\}}. There exists 
\hb{\ve>0} such that 
$$ 
\|\wh{\mfB}_\ka\mfu\|_{\pl W_{\coW\ka}(\tau)} 
\leq c\coc\tau^\ve\sum_{\wk\in\gN_\Ga(\ka)} 
\|u_\wk\|_{W_{\coW\wk}^{2/\mf2}(\tau)}
\qquad\text{$\gK_\Ga$-unif.}
\qb \mfu\in\mfW^{2/\mf2}[\Ga], 
$$ 
and 
$$ 
\|\wh{\mfC}_\ka\mfu\|_{\pl_SW_{\coW\ka}(\tau)} 
\leq c\coc\tau^\ve\sum_{\wk\in\gN_S(\ka)} 
\|u_\wk\|_{\bar W_{\coW\wk}^{2/\mf2}(\tau)}
\qquad\text{$\gK_S$-unif.} 
\qb \mfu\in\bar\mfW^{2/\mf2}, 
\npbd 
$$ 
for 
\hb{0<\tau\leq T}. 
\end{lemma}\po 
\begin{proof} 
Let $s$, $q$, and~$\ve$ be as in the preceding proof. 

\par 
Given any Banach space~$E$, H\"older's inequality gives 
$$ 
\|u\|_{W_{\coW p}^k(J_\tau,E)} 
\leq\tau^\ve\,\|u\|_{W_{\coW q}^k(J_\tau,E)} 
\qa k=0,1. 
$$ 
Hence, by interpolation (see \cite[Theorems VII.2.7.4 
and VII.7.3.4]{Ama19a}), 
\begin{equation}\label{LT.Wpp} 
\|u\|_{W_{\coW p}^{(1-1/p)/2}(J_\tau,E)} 
\leq c\coc\tau^\ve\,\|u\|_{W_{\coW q}^{(1-1/p)/2}(J_\tau,E)} 
\qa 0<\tau<T.  
\end{equation}  

\par 
Set 
$$ 
\cL_\ka^\sa(\tau) 
:=L_q\bigl(J_\tau,W_{\coW\ka}^\sa(\pl\BH^m)\bigr)  
\cap W_{\coW q}^{\sa/2}\bigl(J_\tau,W_{\coW\ka}^0(\pl\BH^m)\bigr) 
\qa \ka\in\gK_\Ga. 
$$ 
Then we get from \eqref{LT.Wpp} 
\begin{equation}\label{LT.L} 
\Vsdot_{\pl_1W_{\coW\ka}(\tau)} 
\leq c\coc\tau^\ve\,\Vsdot_{\cL_\ka^{1-1/p}(\tau)} 
\qquad\text{$\gK_\Ga$-unif.} 
\qb 0<\tau<T.   
\end{equation}  
It is clear from the structure of~$\wh{\mfB}_\ka$ and the mapping 
properties of~$\ga$ that 
$$ 
\|\wh\mfB_\ka\mfu\|_{\cL_\ka^{1-1/p}(\tau)} 
\leq c\,\|\chi\mfu\|_{\cL_\ka^1(\tau)} 
\qquad\text{$\gK_\Ga$-unif.}
$$ 
Since 
\hb{\cL_\ka^s(\tau)\hr\cL_\ka^1(\tau)} \hbox{$\gK_\Ga$-unif.} and unif.\ 
w.r.t.~$\tau$, the first assertion now follows from \eqref{LT.L}, 
Lemma~\ref{lem-LT.ch}, and the fact that $\wh\mfB_\ka$~has its image in the 
closed  linear subspace $\pl_1W_{\coW\ka}(\tau)$ 
of~$\pl W_{\coW\ka}(\tau)$. 

\par 
The proof of the second claim is similar. 
\end{proof}\po  
\section{Proof of Theorem~\ref{thm-L.MR}}\label{sec-P} 
Let 
\hb{\mfE=\prod_{\al\in\sA}E_\al}, where each~$E_\al$ is a Banach space 
and $\sA$~is a countable index set. Then 
$\ell_p(\mfE)$ is the Banach space of \hbox{$p$-summable} 
sequences in~$\mfE$. 

\par 
We put 
$$ 
\BE^i[\al]:=\ell_p\bigl(\mfW^{i/\mf2}[\al]\bigr) 
\qb \ga_0\BE[\al]:=\ell_p\bigl(\ga_0\mfW[\al]\bigr) 
\qa \al\in\{0,\Ga\}, 
$$ 
and 
$$ 
\bar\BE^i:=\ell_p\bigl(\bar\mfW^{i/\mf2}\bigr) 
\qb \ga_0\bar\BE:=\ell_p(\ga_0\bar\mfW) 
$$ 
for 
\hb{i=0,2}. Then 
$$ 
\BE^i:=\BE^i[0]\oplus\BE^i[\Ga]\oplus\bar\BE^i 
\qb i=0,2 
\qa \ga_0\BE:=\ga_0\BE[0]\oplus\ga_0\BE[\Ga]\oplus\ga_0\bar\BE. 
\npbd  
$$ 
Moreover, 
\hb{\BF_\Ga:=\ell_p(\pl\mfW)}, 
\ \hb{\BF_S:=\ell_p(\pl_S\bar\mfW)}.   

\par 
Recall the definitions of $\cRcRc$, $(\cR_\Ga,\cR_\Ga^c)$, and 
$(\cR_S,\cR_S^c)$ in Section~\ref{sec-LT}. 
\setcounter{proposition}{0} 
\begin{proposition}\label{pro-P.R} 
$\cRcRc$ is an \hbox{r-c} pair for $(\BE^k,\bar W^{k/\mf2})$, 
\,\hb{k=0,2}, and 
$$ 
(\cR_\Ga\oplus\cR_S\oplus\cR,\,\cR_\Ga^c\oplus\cR_S^c\oplus\cR^c) 
$$ 
is an \hbox{r-c} pair for 
$$ 
(\BF_\Ga\oplus\BF_S\oplus\ga_0\bar\BE,\,\pl W\oplus\pl_SW\oplus\ga_0\bar W). 
$$ 
\end{proposition}
\begin{proof} 
This follows from \cite[Theorem~6.1]{Ama12b} (also see 
\cite[Theorem~9.3]{Ama12c} or \cite{AmaVolIII21a}). 
\end{proof}\po 
\setcounter{lemma}{1}
\begin{lemma}\label{lem-P.Gr} 
$\ga_0$~is a retraction from~$\bar W^{2/\mf2}$ onto~$\ga_0\bar W$ and 
from~$\BE^2$ onto~$\ga_0\bar\BE$. Furthermore, 
\hb{\ga_0\cR=\cR\ga_0}. 
\end{lemma}
\begin{proof} 
Due to \eqref{C.0} there exists 
\hb{\ga_0^c\in\cL(\ga_0\bar\BE,\BE^2)} such that $(\ga_0,\ga_0^c)$ is 
an \hbox{r-c} pair for~$(\BE^2,\ga_0\bar\BE)$. For the moment we denote it 
by~$(\bar\ga_0,\bar\ga_0^c)$. Then it follows from 
Proposition~\ref{pro-P.R} that 
$$ 
\cR\bar\ga_0\cR^c\in\cL(\bar W^{2/\mf2},\ga_0\bar W) 
\qb \cR\bar\ga_0^c\cR^c\in\cL(\ga_0\bar W,\bar W^{2/\mf2}). 
$$ 
Since $\ga_0$~is the evaluation at 
\hb{t=0} and $\cRcRc$~is independent of 
\hb{t\in J}, we see that 
\hb{\ga_0\cR=\cR\bar\ga_0}. Hence 
\hb{\ga_0\in\cL(\bar W^{2/\mf2},\ga_0\bar W)} and 
\hb{\ga_0^c:=\cR\bar\ga_0^c\cR^c} is a continuous right 
inverse for~$\ga$. 
\end{proof} 

\par 
We write  
$$ 
\wh{\BA}:=\mfsA+\wh{\mfsA} 
\text{ with }  
\wh{\mfsA}:=[\wh{\mfsA}_\ka]_{\ka\in\gK} 
$$ 
and, using \eqref{P.Ws}, set 
\hb{\mfu=(\mfv,\mfw,\mfz)\in\mfsW^{2/\mf2}} and, analogously,   
$$ 
\wh{\BB}\mfu:=(\mfB+\wh{\mfB})\mfw\in\pl\mfW 
\qb \wh{\BC}\mfu:=(\mfC+\wh{\mfC})\mfz\in\pl_S\mfW. 
\npbd 
$$ 
Moreover, 
\hb{\BG:=\BF_\Ga\oplus\BF_S\oplus\ga_0\bar\BE} and 
$[\BG]_{\mfB,\mfC}$~is the  linear subspace  consisting of all 
\hb{(\mfvp,\mf\psi,\mfu_0)} satisfying the compatibility conditions 
$$ 
\mfB(0)\mfu_0 =\mfvp(0) 
\qb \mfC(0)\mfu_0 =\mfpsi(0) 
$$ 
if 
\hb{p>3}, with the corresponding modifications if 
\hb{p<3}. Analogous definitions apply to~$[\BG]_{\wh{\BB},\wh{\BC}}$.
\setcounter{proposition}{2} 
\begin{proposition}\label{pro-P.is} 
${}$ 
\hb{\bigl(\pl_t+\wh{\BA},\,(\wh{\BB},\wh{\BC},\ga_0)\bigr) 
   \in\Lis\bigl(\BE^2,\,\BE^0\oplus[\BG]_{\wh{\BB},\wh{\BC}}\bigr)}. 
\end{proposition}
\begin{proof} 
(1) 
First we prove that 
\begin{equation}\label{P.D} 
\BL:=\bigl(\pl_t+\mfsA,\,(\mfB,\mfC,\ga_0)\bigr) 
\in\Lis\bigl(\BE^2,\,\BE^0\oplus[\BG]_{\mfB,\mfC}\bigr). 
\end{equation}  
Since $\BL$~has diagonal structure, the claim is a direct consequence of 
Propositions \hbox{\ref{pro-MR.1}--\ref{pro-MR.3}}.  

\par 
(2) 
Set 
\hb{\BL_0:=(\wh{\mfsA},\wh{\mfB},\wh{\mfC},0)}. Then  
\begin{equation}\label{P.DD} 
\wh{\BL}:=
(\pl_t+\wh{\BA},\ \wh{\BB},\wh{\BC},\ga_0)=\BL+\wh{\BL}_0. 
\end{equation}  
It follows from \eqref{LT.k} and 
Lemmas \ref{lem-LT.A0} and \ref{lem-LT.BC} that 
\hb{\wh{\BL}_0\in\cL(\BE^2,\,\BE^0\oplus\BG)} and there exists 
\hb{\ve>0} such that 
\begin{equation}\label{P.t} 
\|\wh{\BL}_0\|_{\cL(\BE^2(\tau),\,\BE^0(\tau)\oplus\BG(\tau))} 
\leq c\coc\tau^\ve 
\qa 0<\tau\leq T. 
\end{equation}  
From this and step~(1) we see that 
\begin{equation}\label{P.L} 
\wh{\BL}\in\cL(\BE^2,\,\BE^0\oplus\BG). 
\end{equation}  

\par 
Write 
$$ 
\BE_0^2:=\{\,\mfv\in\BE^2\ ;\ \ga_0\mfv=\mf0\,\} 
\qb \BG_0:=\BF_\Ga\oplus\BF_S\oplus\{0\}. 
$$ 
By Lemma~\ref{lem-P.Gr} we can choose a coretraction~$\ga_0^c$ for 
\hb{\ga_0\in\cL(\BE^2,\ga_0\bar\BE)}. Let 
\hb{(\mff,\mfg)\in\BE^0\oplus[\BG]_{\wh{\BB},\wh{\BC}}} with 
\hb{\mfg=(\mfvp,\mfpsi,\mfu_0)}. Set 
\hb{\ol{\mfu}:=\ga_0^c\mfu_0}. Then 
\hb{\mfu\in\BE^2} satisfies 
\hb{\wh{\BL}\mfu=(\mff,\mfg)} iff 
\hb{\mfv:=\mfu-\ol{\mfu}} is such that 
$$ 
\wh{\BL}\mfv=(\mff,\mfg)-\wh{\BL}\ol{\mfu} 
=:(\mff_0,\mfg_0). 
$$ 
Note that 
\hb{\mfv\in\BE_0^2} and 
\hb{\mfg_0\in[\BG_0]_{\wh{\BB},\wh{\BC}}}. Suppose 
\hb{p>3}. Given any 
\hb{\mfw\in\BE_0^2}, we get 
$$ 
\ga_0(\mfB\mfw)=\mfB(0)\ga_0\mfw=0=\wh{\BB}(0)\ga_0\mfw=\ga_0(\wh{\BB}\mfw). 
$$ 
From this, the analogous relation for $\mfC$ and~$\wh{\BC}$, 
and \eqref{P.L} we infer that it suffices to prove that 
\hb{\wh{\BL}\sco\BE_0^2\ra\BE^0\oplus[\BG_0]_{\mfB,\mfC}} is surjective 
and has a continuous inverse. Obvious modifications apply to 
\hb{p<3}. 

\par 
(3) 
Let 
\hb{\BF:=\BE^0\oplus[\BG_0]_{\mfB,\mfC}} and 
\hb{\mfh\in\BF}. Suppose 
\hb{\mfu\in\BE_0^2} and set 
\hb{\mfv:=\BL\mfu\in\BF}. By step~(1) and \eqref{P.DD}, the equation 
\hb{\wh{\BL}u=\mfh} is equivalent to 
\hb{\mfv+\wh{\BL}_0\BL^{-1}\mfv=\mfh}. Observe 
\hb{\wh{\BL}_0\BL^{-1}\mfv\in\BF}. 

\par 
Due to \eqref{P.t}, we can fix 
\hb{\ol{\tau}\in(0,T]} such that 
\hb{\|\wh{\BL}_0\BL^{-1}\|_{\cL(\BF(\tau))}\leq1/2}. As is well-known 
(e.g.,~\cite[Lemma~12.2]{Ama17a}), this implies that 
\hb{\wh{\BL}\in\Lis\bigl(\BF(\ol{\tau}),\BF(\ol{\tau})\bigr)}. 

\par 
(4) 
If 
\hb{\ol{\tau}=T}, then we are done. Otherwise we repeat this argument 
for the problem in 
\hb{[0,\,T-\ol{\tau}]} obtained by the time-shift 
\hb{t\mt t-\ol{\tau}} and with the initial value~$\mfu(\ol{\tau})$. After 
finitely many such steps we reach~$T$. The proposition is proved. 
\end{proof} 
\begin{proof} 
{\it of Theorem~\ref{thm-L.MR}} 
Now we write $(\BA,\BB,\BC)$ for $(\wh{\BA},\wh{\BB},\wh{\BC})$ 
if 
$(\wh{\mfsA},\wh{\mfB},\wh{\mfC})$ equals $(\mfsA^0,\mfB^0,\mfC^0)$, and 
$(\wt{\BA},\wt{\BB},\wt{\BC})$ otherwise. We also set 
\hb{G:=\pl W\oplus\pl_SW\oplus\ga_0\bar W}. Then Proposition~\ref{pro-P.R} 
implies 
$$ 
\vecTeX\cR:=\cR\oplus(\cR_\Ga\oplus\cR_S\oplus\cR) 
\in\cL(\BE^k\oplus\BG,\,\bar W^{k/\mf2}\oplus G) 
$$ 
and 
$$ 
\ {}
\vecTeX\cR^c:=\cR^c\oplus(\cR_\Ga^c\oplus\cR_S^c\oplus\cR^c) 
\in\cL(\bar W^{k/\mf2}\oplus G,\,\BE^k\oplus\BG) 
\npbd 
$$ 
for 
\hb{k=0,2}.\po 

\par 
Let 
\hb{\bigl(\mfu,(\mfvp,\mfpsi,\mfu_0)\bigr)\in\BE^2\oplus\BG} and write 
\hb{\bigl(u,(\vp,\psi,u_0)\bigr)}
for its image under~$\vecTeX\cR$. Then we obtain from \eqref{LT.RBB} and 
\eqref{LT.CCR}  
$$ 
\cB u=\cB\cR\mfu=\cR_\Ga\BB\mfu 
\qb \cC u=\cC\cR\mfu=\cR_S\BC\mfu. 
$$ 
Suppose 
\hb{p>3}, 
\,\hb{\mfu_0=\ga_0\mfu}, and 
\hb{\BB(0)\mfu_0=\mfvp(0)}. Then 
$$ 
\bal 
\vp(0) 
&=\cR_\Ga\mfvp(0) 
 =\cR_\Ga\BB(0)\mfu_0 
 =\cR_\Ga\ga_0(\BB\mfu)\cr 
&=\ga_0\cR_\Ga\BB\mfu 
 =\ga_0\cB\cR\mfu=\ga_0(\cB u) 
 =\cB(0)(\ga_0u). 
\eal 
$$ 

\par 
Since Lemma~\ref{lem-P.Gr} and 
\hb{\mfu_0=\ga_0\mfu} imply 
\hb{u_0=\ga_0u}, we see that 
\hb{\cB(0)u_0=\vp_0}. Similarly, we find that it follows from 
\hb{\BC(0)\mfu_0=\mfpsi(0)} that 
\hb{\cC(0)u_0=\psi(0)}. Thus, letting 
\hb{\BF:=[\BG]_{\BB,\BC}} and 
\hb{F:=[G]_{\cB,\cC}}, 
we have shown that 
$$ 
\vecTeX\cR(\BE^2\oplus\BF)\is\bar W^{2/\mf2}\oplus F. 
$$ 
Consequently, 
\begin{equation}\label{P.R} 
\vecTeX\cR\in\cL(\BE^2\oplus\BF,\,\bar W^{2/\mf2}\oplus F). 
\end{equation}  
We find analogously that 
\begin{equation}\label{P.Rc} 
\vecTeX\cR^c\in\cL(\bar W^{2/\mf2}\oplus F,\,\BE^2\oplus\BF). 
\npbd 
\end{equation}  
This holds for 
\hb{p>3}. The case 
\hb{p<3} is similar.\po 

\par 
Now we set 
$$ 
L:=\bigl(\pl_t+\cA,\,(\cB,\cC,\ga_0)\bigr) 
\qb \BL:=\bigl(\pl+\BA,\,(\BB,\BC,\ga_0)\bigr). 
$$ 
Define 
\hb{\wt{\BL}:=\bigl(\pl+\wt{\BA},\,(\wt{\BB},\wt{\BC},\ga_0)\bigr)}. 
It is a consequence 
of \eqref{LT.RAA}, \eqref{LT.AAR}, \eqref{LT.RBB}, \eqref{LT.BBR}, 
\eqref{LT.CCR}, and the fact that $\cRcRc$ is independent of~$t$ that 
\begin{equation}\label{P.LR} 
L\cR=\vecTeX\cR\BL 
\qb \vecTeX\cR^cL=\wt{\BL}\vecTeX\cR^c. 
\end{equation}  
Proposition~\ref{pro-P.is} guarantees  
\begin{equation}\label{P.LL} 
\BL,\wt{\BL}\in\Lis(\BE^2,\,\BE^0\oplus\BF). 
\end{equation}  
Suppose 
\hb{Lu=0}. Then \eqref{P.LR} and \eqref{P.LL} imply 
\hb{\cR^cu=0}. Thus 
\hb{u=\cR\cR^cu=0}. This shows that $L$~is injective. 

\par 
Let 
\hb{(f,g)\in\bar W^{0/\mf2}\oplus F}. By \eqref{P.LL} we find  
\hb{\mfu\in\BE^2} satisfying 
\hb{\BL\mfu=\vecTeX\cR^c(f,g)}. Put 
$$ 
u:=\cR\mfu=\cR\BL^{-1}\vecTeX\cR^c(f,g)\in\bar W^{2/\mf2}. 
$$ 
Then, by \eqref{P.LR} and \eqref{P.LL}, 
$$ 
Lu=L\cR\BL^{-1}\vecTeX\cR^c(f,g) 
=\vecTeX\cR\BL\BL^{-1}\vecTeX\cR^c(f,g) 
=\vecTeX\cR\vecTeX\cR^c(f,g) 
=(f,g). 
$$ 
Hence $L$~is surjective and, by \eqref{P.R} and \eqref{P.Rc}, 
$$ 
L^{-1}=\cR\BL^{-1}\vecTeX\cR^c 
\in\cL(\bar W^{0/\mf2}\oplus F,\,\bar W^{2/\mf2}). 
\npbd  
$$ 
The proof is accomplished. 
\end{proof}
\setcounter{remarks}{3} 
\begin{remarks}\label{rem-P.P}  
(a) 
Recall that either some of $\Ga_0$, $\Ga_1$, and~$S$, or all of them, 
can be empty. If 
\hb{(\Ga,S)\neq\{\es,\es\}}, then the result is new. Otherwise, it is 
a special case of the more general Theorem~1.23(ii) of~\cite{Ama17a}. 

\par 
(b) 
Theorem~\ref{thm-L.MR} is true for systems 
\hb{u=(u^1,\ldots,u^n)}, provided the uniform 
Lopatinskii-Shapiro conditions apply. This is trivially the case if 
$a$~is a diagonal matrix.\qed 
\end{remarks} 
\section{Membranes with Boundary}\label{sec-MB} 
Now we turn to the case of membranes intersecting~$\Ga$ transversally. This 
case is handled by reducing it to the situation studied in the preceding 
section. 
\begin{theorem}\label{thm-MB.P} 
Let \eqref{S.ass} be satisfied. Theorem~\ref{thm-L.MR} applies with $\Mg$ 
replaced by $\whMwhg$. 
\end{theorem} 
\begin{proof} 
Theorem~\ref{thm-S.M}. 
\end{proof}\po  
In preparation for the proof of Theorem~\ref{thm-R.RD} we derive rather 
explicit representations of $(\cA,\cB,\cC)$ and the relevant function 
spaces in a tubular neighborhood of 
\hb{\Sa=\pl S} in~$\wh{M}$. 

\par 
We use the notations of Sections \ref{sec-H} and~\ref{sec-S} and set 
$$ 
U:=U(\ve) 
\qa \dot N:=N(\ve)\big\backslash\bigl\{(0,0)\bigr\} 
\qb \wt{g}:=g_2/\rho^2. 
$$ 
The Christoffel symbols~$\wt{\Ga}_{ij}^k$ for the metric~$\wt{g}$ 
turn out to be  
$$ 
\wt{\Ga}_{11}^1=\wt{\Ga}_{12}^2=-\wt{\Ga}_{22}^1=-\rho^{-1}\pl_1\rho 
\qb \wt{\Ga}_{22}^2=\wt{\Ga}_{12}^1=-\wt{\Ga}_{11}^2=-\rho^{-1}\pl_2\rho. 
$$ 
We set        
\hb{D:=(D_1,D_2)} with 
\hb{D_i:=\rho\pl_i} and 
\hb{\wt{\na}:=\na_{\wt{g}}}. Then, for 
\hb{u\in\bar C^2(\dot N\ssm G_\sa)} with 
\hb{G_\sa:=\graph(f_\sa)}, 
\begin{equation}\label{MB.N} 
\bal 
\rho^2\wt{\na}_{11}u 
&=D_1^2u-\pl_2\rho D_2u,    
    &\quad \rho^2\wt{\na}_{12}u 
    &=D_1D_2u+\pl_2\rho D_1u,\cr 
\rho^2\wt{\na}_{21}u 
&=D_2D_1u+\pl_1\rho D_2u, 
    &\quad \rho^2\wt{\na}_{22}u 
    &=D_2^2u-\pl_1\rho D_1u. 
\eal 
\end{equation}  
Hence, see \eqref{F.1} and \eqref{F.2}, 
\begin{equation}\label{MB.D1} 
|\wt{\na}u|_{\wt{g}^*}^2 
=|Du|^2=|D_1u|^2+|D_2u|^2 
\end{equation}  
and 
\begin{equation}\label{MB.D2} 
|\wt{\na}^2u|_{\wt{g}_0^2}^2 
=\rho^4\Bigl((\wt{\na}_{11}u)^2+(\wt{\na}_{12}u)^2 
 +(\wt{\na}_{21}u)^2+(\wt{\na}_{22}u)^2\Bigr). 
\end{equation}  

\par 
Let 
$$ 
[D^2u]^2:=(D_1^2u)^2+(D_1D_2u)^2+(D_2D_1u)^2+(D_2^2u)^2 
$$ 
and 
$$ 
\dl u\dr_\rho^2 
:=|u|^2+|Du|^2+[D^2u]^2+|\na_hu|_{h^*}^2+|\na_h^2u|_{h_0^2}^2. 
$$ 
We write 
$$ 
\bar\cW_{\coW p}^2 
=\bigl(\bar\cW_{\coW p}^2,\Vsdot_{\bar\cW_{\coW p}^2}\bigr) 
$$ 
for the space of all 
\hb{u\in L_{1,\loc}(\dot N\times\Sa)} for which the norm 
$$ 
\|u\|_{\bar\cW_{\coW p}^2} 
:=\Bigl(\int_\Sa\int_{\dot N\setminus G(\sa)} 
\dl u\dr_\rho^p(x,y,\sa)\,\frac{d(x,y)}{\rho^2(x,y)} 
\,d\vol_\Sa(\sa)\Bigr)^{1/p} 
$$ 
is finite. Moreover, $\bar\cW_{\coW p}^0$~is obtained by replacing~%
\hb{\dl\cdot\dr_\rho} by~%
\hb{\vsdot}. It is then clear how to define the anisotropic 
spaces~$\bar\cW_{\coW p}^{k/\mf2}$, 
\,\hb{k=0,2}. 
\setcounter{proposition}{1} 
\begin{proposition}\label{pro-MB.N} 
${}$
\hb{u\in\bar W_{\coW p}^2(U\ssm S;\wh{g})} iff 
\hb{\chi_*u\in\bar\cW_{\coW p}^2}. 
\end{proposition}
\begin{proof} 
First we note that  
\begin{equation}\label{MB.W} 
\sqrt{\wt{g}}=\rho^{-2}.  
\end{equation}  
It follows from \eqref{MB.D2} that 
$$ 
|\wt{\na}^2v|_{\wt{g}_0^2}^2 
\leq c\bigl(|Dv|^2+[D^2v]^2\bigr). 
$$ 
Consequently, since 
\hb{\na=\chi^*(\wt{\na}\oplus\naSa)\chi_*}, 
\begin{equation}\label{MB.u} 
\|u\|_{\bar W_{\coW p}^2(U\setminus S,\,\wh{g})} 
\leq c\,\|\chi_*u\|_{\bar\cW_p^2}. 
\end{equation}  
If $(u_j)$~is a converging sequence in 
\hb{\bar W_{\coW p}^2:=\bar W_{\coW p}^2(U\ssm S;\wh{g})}, 
then we infer from \hbox{\eqref{MB.N}--\eqref{MB.D2}} 
that 
$(\chi_*u_j)$ converges in~$\bar\cW_{\coW p}^2$. Hence 
\hb{\chi_*u\in\bar\cW_{\coW p}^2} if 
\hb{u\in\bar W_{\coW p}^2}. From this, \eqref{MB.u}, and 
Banach's homomorphism theorem we obtain 
$$ 
\|\chi_*\!\cdot\|_{\bar\cW_{\coW p}^2} 
\sim\Vsdot_{\bar W_{\coW p}^2}. 
\npbd 
$$ 
This proves the claim. 
\end{proof} 
Let 
\hb{\lda\sco V_\lda\ra(-1,1)^{m-2}}, 
\ \hb{\sa\mt z} be a local chart for~$\Sa$. Then 
$$ 
\ka:=(\id_{\dot N}\times\lda)\circ\chi\sco U_{\coU\ka}\ra Q_\ka^m 
\qb q\mt(x,y,z)=(x^1,x^2,\ldots,x^m) 
$$ 
is a local chart for~$\wh{M}$ with 
\hb{U_{\coU\ka}\is U} and 
\hb{\ka(U_{\coU\ka})=\dot N\times(-1,1)^{m-2}}. Set 
\hb{\wt{h}:=\lda_*h}. It follows from \eqref{S.g} that 
\hb{\ka_*\wh{g}=\wt{g}+\wt{h}}.

\par 
Assume 
\hb{u\in C^2(\wh{M})}, put 
\hb{v:=\ka_*u\in C^2(Q_\ka^m)}, and denote by 
\hb{e_1,\ldots,e_m} the standard basis of~$\BR^m$. Then 
\begin{equation}\label{MB.gr} 
\ka_*(\grad_{\wh{g}}u) 
=\grad_{\wt{g}}v+\grad_{\wt{h}}v 
=\rho^2(\pl_x ve_1+\pl_yve_2)+\grad_{\wt{h}}v. 
\end{equation}  
Similarly, let 
\hb{X=X^i\pl/\pl x^i\in C^1(TU_{\coU\ka})} and set 
$$ 
Y:=\ka_*X=Y^ie_i=Y^1e_1+Y^2e_2+\wt{Y}, 
$$ 
where 
\hb{\wt{Y}=Y^\al e_\al} with $\al$~running from~$3$ to~$m$. Observe that $v$ 
and~$Y$ depend on $(x,y,z)$. It follows 
$$ 
\ka_*(\tdiv_{\wh{g}}X) 
=\ka_*\Bigl(\frac1{\sqrt{\wh{g}}}\,\pl_i\bigl(\sqrt{\wh{g}}\,X^i\bigr)\Bigr) 
=\frac1{\sqrt{\ka_*\wh{g}}}\,\pl_i\bigl(\sqrt{\ka_*\wh{g}}\,Y^i\bigr). 
$$ 
Since 
\hb{\sqrt{\ka_*\wh{g}}=\sqrt{\wt{g}}\,\sqrt{\wt{h}} 
   =\rho^{-2}\sqrt{\wt{h}}}, we see that 
$$ 
\ka_*(\tdiv_{\wh{g}}X) 
=\rho^2\bigl(\pl_x(\rho^{-2}Y^1)+\pl_y(\rho^{-2}Y^2)\bigr) 
+\tdiv_{\wt{h}}\wt{Y}. 
$$ 
From this and \eqref{MB.gr} we obtain, letting 
\hb{\wt{a}:=\ka_*a}, 
\begin{equation}\label{MB.kA} 
\bal 
\ka_*(\cA u) 
&=-\ka_*\bigl(\tdiv_{\wh{g}}(a\grad_{\wh{g}}u)\bigr)\cr 
&=-\rho^2\bigl(\pl_x(\wt{a}\pl_xv)+\pl_y(\wt{a}\pl_yv)\bigr) 
 -\tdiv_{\wt{h}}(\wt{a}\grad_{\wt{h}}v).  
\eal 
\end{equation}  
As in \eqref{R.nu}, we define curvilinear derivatives by 
$$ 
\pl_\nu u:=\ka^*\pl_x(\ka_*u) 
\qb \pl_\mu u:=\ka^*\pl_y(\ka_*u). 
$$ 
Then it follows from \eqref{MB.kA} that 
\begin{equation}\label{MB.A} 
\cA_Uu=-\rho^2\bigl(\pl_\nu(a\pl_\nu u)+\pl_\mu(a\pl_\mu u)\bigr) 
-\tdiv_\Sa(a\grad_\Sa u), 
\end{equation}  
where $\cA_U$~is the restriction of~$\cA$ to~$U$. Due 
to~\eqref{F.1}, the regularity assumption for~$a$ stipulated in 
Theorem~\ref{thm-R.RD} means that 
\hb{a\in\bar{BC}^{1/\mf2}(\wh{M}\ssm S\times J,\ \wh{g}+dt^2)}. 

\par 
Analogously, 
$$ 
\ka_*(\cB^1u) 
=\ka_*\bigl(\nu\bsn\ga(a\grad_{\wh{g}}u)\bigr)_{\wh{g}} 
=\Bigl(\ka_*\nu\Bsn 
 \ga\bigl(\wt{a}(\grad_{\wt{g}}v 
 +\grad_{\wt{h}}v)\bigr)\Bigr)_{\wt{g}+\wt{h}}. 
$$ 
From \eqref{C.nu} and 
\hb{(\ka_*\wh{g})^{1i}=0} for 
\hb{i\geq2} we deduce that 
\hb{\ka_*\nu(0,y,z)=\rho(0,y)e_1}. 
Hence 
$$ 
\ka_*(\cB^1u)(y,z) 
=\rho(0,y)\wt{a}(0,y,z)\pl_xv(0,y,z). 
$$ 
This implies 
\begin{equation}\label{MB.B} 
\cB_U^1u=\ga(\rho a)\pl_\nu u 
\text{ on }\Ga\cap U. 
\end{equation}  

\par 
Next we determine the first order transmission operator on~$\whMwhg$. 
Recalling definition~\eqref{H.f}, we set 
\hb{s:=\chi^*f\in C^\iy(U)}. 
\setcounter{proposition}{2} 
\begin{proposition}\label{pro-MB.S} 
Define 
\begin{equation}\label{MB.nu} 
(\nu_S^1,\nu_S^2,\nu_S^3) 
:=(\pl_\nu s,-1,1) 
\Big/\sqrt{1+(\pl_\nu s)^2+|\grad_\Sa s|_\Sa^2}. 
\end{equation}  
Then 
$$ 
\cC_U^1u 
=\eea{a\pl_{\nu_S}u} 
=\beea{a\bigl(\nu_S^1\pl_\nu u+\nu_S^2\pl_\mu u 
 +\nu_S^3(\grad_\Sa u\sn\grad_\Sa s)_\Sa\bigr)}. 
$$ 
\end{proposition}
\begin{proof} 
(1) It follows from \eqref{H.f} that 
$$ 
\wt{S}:=\chi(S) 
=\bigl\{\,(x,f_\sa(x),\sa)\ ;\ 0<x<\ve,\ \sa\in\Sa\,\bigr\}. 
$$ 
We write 
\hb{f(x,z):=f_{\lda^{-1}(z)}(x)}. Then 
\hb{(x,z)\mt F(x,z):=\bigl(x,f(x,z),z\bigr)} is a local 
parametrization of~$\wt{S}$, and 
$$ 
\pl F:=[\pl_jF^i] 
=\left[ 
\begin{array}{ccc} 
1       &\  &0\\ 
\pl_xf  &\  &\pl_zf\\ 
0       &\  &1_{m-2} 
\end{array} 
\right] 
\in\BR^{m\times(m-1)}. 
$$ 
Hence, given 
\hb{(\bar x,\bar z)\in(0,\ve)\times(-1,1)^{m-2}}, 
$$ 
T_{(\bar x,\bar z)}\wt{S} 
=\bigl\{(\bar x,\bar z)\bigr\}\times\pl F(\bar x,\bar z)\BR^{m-1} 
\is T_{(\bar x,\bar z)}\BR^m. 
$$ 
For 
\hb{\Xi:=(\xi,\za)\in\BR\times\BR^{m-2}} and 
\hb{\wh{\Xi}:=(\wh{\xi},\wh{\eta},\wh{\za})\in\BR\times\BR\times\BR^{m-2}} 
we find 
$$ 
\bigl(\wh{\Xi}\bsn\pl F(\bar x,\bar z)\Xi\bigr)_{\wt{g}+\wt{h}} 
=\al(\bar x,\bar z) 
\Bigl(\wh{\xi}\xi+\wh{\eta}\bigl(\pl_xf(\bar x,\bar z)\xi 
+\pl_zf(\bar x,\bar z)\za\bigr)\Bigr)_{\BR^2}+(\wh{\za}\sn\za)_{\wt{h}}, 
$$ 
where 
\hb{\al(\bar x,\bar z):=\rho^{-2}\bigl(\bar x,f(\bar x,\bar z)\bigr)}. 
Choose 
\hb{\wh{\xi}:=\pl_xf(\bar x,\bar z)}, 
\ \hb{\wh{\eta}:=-1}, and 
$\wh{\za}$ in $\BR^{m-2}$ such that 
\hb{(\wh{\za}\sn\wt{\za})_{\wt{h}} 
   =\bigl\dl\pl_zf(\bar x,\bar z),\wt{\za}\bigr\dr_{\BR^{m-2}}} for all 
\hb{\wt{\za}\in\BR^{m-2}}, that is, 
$\wh{\za}$~equals $\grad_{\wt{h}}f(\bar x,\bar z)$. Then 
\hb{\bigl((\bar x,\bar z),\wh{\Xi}\bigr)\perp T_{(\bar x,\bar z)}\wt{S}}. 
Now we define 
$$ 
(\wt{\nu}^1,\wt{\nu}^2,\wt{\nu}^3) 
:=(\pl_xf,-1,1)\Big/
\sqrt{1+(\pl_xf)^2+|\grad_{\wt{h}}f|_{\wt{h}}^2}. 
$$ 
Then 
$$ 
\wt{\nu}:=\wt{\nu}^1e_1+\wt{\nu}^2e_2+\wt{\nu}^3\grad_{\wt{h}}f 
$$ 
is the positive\footnote{By the conventions employed in~\eqref{S.nu}.}  
normal on~$\wt{S}$, since 
\hb{\wt{\nu}^1(\bar x,\bar z)e_1+\wt{\nu}^2(\bar x,\bar z)e_2} is a 
positive multiple of the positive normal of the graph of~$f_{\bar\sa}$ at 
\hb{\bigl(\bar x,f_{\bar\sa}(\bar x)\bigr)}, where 
\hb{\bar\sa:=\lda^{-1}(\bar z)} (cf.~\eqref{S.nu}). 

\par 
(2) 
Using \eqref{MB.gr} and 
\hb{\ka_*\nu_S=\wt{\nu}}, we obtain 
$$ 
\bal 
\ka_*(\nu_S\sn\grad_{\wh{g}}u)_{\wh{g}} 
&=(\ka_*\nu_S\sn\grad_{\wt{g}}v+\grad_{\wt{h}}v)_{\wt{g}+\wt{h}}\cr 
&=\wt{\nu}^1\pl_xv+\wt{\nu}^2\pl_yv 
 +\wt{\nu}^3(\grad_{\wt{h}}f\sn\grad_{\wt{h}}v)_{\wt{h}}\cr 
&=\ka_*\bigl(\nu_S^1\pl_\nu u+\nu_S^2\pl_\mu u 
 +\nu_S^3(\grad_\Sa s\sn\grad_\Sa u)_\Sa\bigr), 
\eal 
$$ 
the $\nu_S^i$~being given by \eqref{MB.nu}. From this the assertion is now 
clear. 
\end{proof} 

\par 
Note that, see~\eqref{MB.W},  
\begin{equation}\label{MB.Ga} 
\chi_*(d\vol_{\Ga\cap U})\sim\frac{dy}{y^2}\,d\vol_\Sa. 
\end{equation}  
Since, by~\eqref{H.f}, 
$$ 
\wt{g}_{G(\sa)}(\tau) 
=\frac{1+(\pl f_\sa(x))^2}{x^2+(f_\sa(x))^2}\,dx^2\sim\frac{dx^2}{x^2} 
$$ 
uniformly with respect to 
\hb{\tau=\bigl(x,f_\sa(x)\bigr)\in G_\sa} and 
\hb{\sa\in\Sa}, we see that 
\begin{equation}\label{MB.S}  
\chi_*(d\vol_{S\cap U})\sim\frac{dx}{x^2}\,d\vol_\Sa. 
\end{equation}  

\par 
On the basis of \eqref{MB.Ga} and \eqref{MB.S} it is possible to represent 
the trace spaces on 
\hb{\Ga\ssm\Sa} and 
\hb{S\ssm\Sa} analogously to~$\bar\cW_{\coW p}^{2/\mf2}$. 
Details are left to the reader. 
\begin{proof}
{\it of Theorem~\ref{R.RD}} 
The claim is an easy consequence of Theorems \ref{thm-H.C} and 
\ref{thm-L.MR}, \eqref{MB.D1}, \eqref{MB.D2}, Proposition~\ref{pro-MB.N}, 
\eqref{MB.A}, \eqref{MB.B}, and Proposition~\ref{pro-MB.S}.   
\end{proof}

{\footnotesize

\noindent 
Herbert Amann, 
Math.\ Institut, Universit\"at Z\"urich, Winterthurerstr.~190,\\   
CH 8057 Z\"urich, Switzerland, herbert.amann@math.uzh.ch
}

\begin{thebibliography}{10}
\def\cprime{$'$} \def\polhk#1{\setbox0=\hbox{#1}{\ooalign{\hidewidth
  \lower1.5ex\hbox{`}\hidewidth\crcr\unhbox0}}}

\bibitem{ARW19a}
H.~Abels, M.~Rauchecker, M.~Wilke.
\newblock Well-posedness and qualitative behaviour of the {M}ullins--{S}ekerka
  problem with ninety-degree angle boundary contact.
\newblock  (2019).
\newblock arXiv:1902.03611.

\bibitem{ADN64a}
S.~Agmon, A.~Douglis, L.~Nirenberg.
\newblock Estimates near the boundary for solutions of elliptic partial
  differential equations satisfying general boundary conditions {II}.
\newblock {\em Comm.\ Pure Appl.\ Math.}, {\bf 17} (1964), 35--92.

\bibitem{Ama84b}
H.~Amann.
\newblock Existence and regularity for semilinear parabolic evolution
  equations.
\newblock {\em Ann.\ Scuola Norm.\ Sup.\ Pisa, Ser.~IV}, {\bf 11} (1984),
  593--676.

\bibitem{Ama85a}
H.~Amann.
\newblock Global existence for semilinear parabolic systems.
\newblock {\em J.~reine angew.\ Math.}, {\bf 360} (1985), 47--83.

\bibitem{Ama90a}
H.~Amann.
\newblock Dynamic theory of quasilinear parabolic equations. {II}.
  {R}eaction-diffusion systems.
\newblock {\em Differential Integral Equations}, {\bf 3}(1) (1990), 13--75.

\bibitem{Ama93a}
H.~Amann.
\newblock Nonhomogeneous linear and quasilinear elliptic and parabolic boundary
  value problems.
\newblock In {\em Function spaces, differential operators and nonlinear
  analysis (Friedrichroda, 1992)}, pages 9--126. \emph{Teubner-Texte
  Math.,}~{\bf 133}, Stuttgart, 1993.

\bibitem{Ama95a}
H.~Amann.
\newblock {\em Linear and quasilinear parabolic problems. {V}ol. {I} Abstract
  linear theory}.
\newblock Birk\-h\"{a}u\-ser, Basel, 1995.

\bibitem{Ama04a}
H.~Amann.
\newblock Maximal regularity for nonautonomous evolution equations.
\newblock {\em Adv.\ Nonl.\ Studies}, {\bf 4} (2004), 417--430.

\bibitem{Ama12c}
H.~Amann.
\newblock Anisotropic function spaces on singular manifolds.
\newblock  (2012).
\newblock arXiv:1204.0606.

\bibitem{Ama12b}
H.~Amann.
\newblock Function spaces on singular manifolds.
\newblock {\em Math. Nachr.}, {\bf 286} (2012), 436--475.

\bibitem{Ama15a}
H.~Amann.
\newblock Uniformly regular and singular {R}iemannian manifolds.
\newblock In {\em Elliptic and parabolic equations}, volume 119 of {\em
  Springer Proc. Math. Stat.}, pages 1--43. Springer, Cham, 2015.

\bibitem{Ama17a}
H.~Amann.
\newblock Cauchy problems for parabolic equations in {S}obolev-{S}lobodeckii
  and {H}\"{o}lder spaces on uniformly regular {R}iemannian manifolds.
\newblock {\em J. Evol. Equ.}, {\bf 17}(1) (2017), 51--100.

\bibitem{Ama19a}
H.~Amann.
\newblock {\em Linear and quasilinear parabolic problems. {V}ol. {II} Function
  spaces}.
\newblock Birk\-h\"{a}u\-ser, Basel, 2019.


\bibitem{Ama20a}
H.~Amann.
\newblock Linear parabolic equations with strong boundary degenerations.
\newblock {\em J.~Elliptic Parabolic Equ.}, {\bf 6} (2020), 123--144.

\bibitem{AmaVolIII21a}
H.~Amann.
\newblock {\em Linear and quasilinear parabolic problems. {V}ol. {III}
  Differential equations}.
\newblock Birk\-h\"au\-ser, Ba\-sel, 2021.
\newblock In preparation.

\bibitem{AHS94a}
H.~Amann, M.~Hieber, G.~Simonett.
\newblock Bounded ${H}_\infty$-calculus for elliptic operators.
\newblock {\em Diff.\ Int.\ Equ.}, {\bf 7} (1994), 613--653.

\bibitem{AGN19a}
B.~Ammann, N.~Gro{\ss}e, V.~Nistor.
\newblock Analysis and boundary value problems on singular domains: an approach
  via bounded geometry.
\newblock {\em C. R. Math. Acad. Sci. Paris}, {\bf 357}(6) (2019), 487--493.

\bibitem{AGN19b}
B.~Ammann, N.~Gro{\ss}e, V.~Nistor.
\newblock The strong {L}egendre condition and the well-posedness of mixed
  {R}obin problems on manifolds with bounded geometry.
\newblock {\em Rev. Roumaine Math. Pures Appl.}, {\bf 64}(2-3) (2019), 85--111.

\bibitem{AGN19c}
B.~Ammann, N.~Gro{\ss}e, V.~Nistor.
\newblock Well-posedness of the {L}aplacian on manifolds with boundary and
  bounded geometry.
\newblock {\em Math. Nachr.}, {\bf 292}(6) (2019), 1213--1237.

\bibitem{Bro59a}
F.E. Browder.
\newblock Estimates and existence theorems for elliptic boundary value
  problems.
\newblock {\em Proc. Nat. Acad. Sci. U.S.A.}, {\bf 45} (1959), 365--372.

\bibitem{BMNZ10a}
C.~B\u{a}cu\c{t}\u{a}, A.L. Mazzucato, V.~Nistor, L.~Zikatanov.
\newblock Interface and mixed boundary value problems on {$n$}-dimensional
  polyhedral domains.
\newblock {\em Doc. Math.}, {\bf 15} (2010), 687--745.

\bibitem{DHP03a}
R.~Denk, M.~Hieber, J.~Pr{\"u}ss.
\newblock {$\mathcal R$}-boundedness, {F}ourier multipliers and problems of
  elliptic and parabolic type.
\newblock {\em Mem. Amer. Math. Soc.}, {\bf 166}(788) (2003).

\bibitem{DHP07a}
R.~Denk, M.~Hieber, J.~Pr{\"u}ss.
\newblock Optimal {$L\sp p$}-{$L\sp q$}-estimates for parabolic boundary value
  problems with inhomogeneous data.
\newblock {\em Math. Z.}, {\bf 257}(1) (2007), 193--224.

\bibitem{DSS16a}
M.~Disconzi, Y.~Shao, G.~Simonett.
\newblock Remarks on uniformly regular {R}iemannian manifolds.
\newblock {\em Math. Nachr.}, {\bf 289} (2016), 232--242.

\bibitem{GaR19a}
H.~Garcke, M.~Rauchecker.
\newblock Stability analysis for stationary solutions of the
  {M}ullins--{S}ekerka flow with boundary contact.
\newblock  (2019).
\newblock arXiv:1907.00833.

\bibitem{GSchn13a}
N.~Gro{\ss}e, C.~Schneider.
\newblock Sobolev spaces on {R}iemannian manifolds with bounded geometry:
  general coordinates and traces.
\newblock {\em Math. Nachr.}, {\bf 286}(16) (2013), 1586--1613.

\bibitem{LSU68a}
O.A. Ladyzhenskaya, V.A. Solonnikov, N.N. Ural'ceva.
\newblock {\em Linear and Quasilinear Equations of Parabolic Type}.
\newblock Amer.\ Math.\ Soc., Transl.\ Math.\ Monographs, Providence, R.I.,
  1968.

\bibitem{LaW20a}
Ph. Lauren\c{c}ot, Ch. Walker.
\newblock Shape {D}erivative of the {D}irichlet {E}nergy for a {T}ransmission
  {P}roblem.
\newblock {\em Arch. Ration. Mech. Anal.}, {\bf 237}(1) (2020), 447--496.

\bibitem{LMN10a}
H.~Li, A.L. Mazzucato, V.~Nistor.
\newblock Analysis of the finite element method for transmission/mixed boundary
  value problems on general polygonal domains.
\newblock {\em Electron. Trans. Numer. Anal.}, {\bf 37} (2010), 41--69.

\bibitem{LNQ13a}
H.~Li, V.~Nistor, Y.~Qiao.
\newblock Uniform shift estimates for transmission problems and optimal rates
  of convergence for the parametric finite element method.
\newblock In {\em Numerical analysis and its applications}, volume 8236 of {\em
  Lecture Notes in Comput. Sci.}, pages 12--23. Springer, Heidelberg, 2013.

\bibitem{MaNi10a}
A.L. Mazzucato, V.~Nistor.
\newblock Well-posedness and regularity for the elasticity equation with mixed
  boundary conditions on polyhedral domains and domains with cracks.
\newblock {\em Arch. Ration. Mech. Anal.}, {\bf 195}(1) (2010), 25--73.

\bibitem{PrS16a}
J.~Pr{\"u}ss, G.~Simonett.
\newblock {\em Moving {I}nterfaces and {Q}uasilinear {P}arabolic {E}volution
  {E}quations}, volume 105 of {\em Monographs in Mathematics}.
\newblock Birk\-h\"au\-ser, Basel, 2016.

\bibitem{PSW19a}
J.~Pr\"{u}ss, G.~Simonett, M.~Wilke.
\newblock The {R}ayleigh--{T}aylor instability for the {V}erigin problem with
  and without phase transition.
\newblock {\em NoDEA Nonlinear Differential Equations Appl.}, {\bf 26}(3)
  (2019), Paper No.~18, 35.

\bibitem{Rau20a}
M.~Rauchecker.
\newblock Strong solutions to the {S}tefan problem with {G}ibbs--{T}homson
  correction and boundary contact.
\newblock  (2020).
\newblock arXiv:2001.06438.

\bibitem{Schi01a}
Th. Schick.
\newblock Manifolds with boundary and of bounded geometry.
\newblock {\em Math. Nachr.}, {\bf 223} (2001), 103--120.

\bibitem{Tay96a}
M.E. Taylor.
\newblock {\em Partial differential equations {I}. {B}asic theory}.
\newblock Springer-Verlag, New York, 1996.

\bibitem{Wil17a}
M.~Wilke.
\newblock Rayleigh--{T}aylor instability for the two-phase {N}avier--{S}tokes
  equations with surface tension in cylindrical domains.
\newblock  (2017).
\newblock arXiv:1703.05214.

\bibitem{WRL95a}
J.T. Wloka, B.~Rowley, B.~Lawruk.
\newblock {\em Boundary {V}alue {P}roblems for {E}lliptic {S}ystems}.
\newblock Cambridge University Press, Cambridge, 1995.
\end{thebibliography}
\end{document}